\documentclass[a4paper,11pt]{article}
\usepackage{amsmath,indentfirst,amsfonts,dsfont,amssymb,amsthm,graphicx,enumerate,mathrsfs,color,bm,multirow,booktabs,subfigure}
\usepackage{algorithm}               
\usepackage{algcompatible}             
\usepackage{url}
\usepackage[dvips]{epsfig}
\usepackage{epstopdf}
\usepackage{amsmath,amsfonts,amssymb}

\usepackage{hyperref}
\hypersetup{hypertex=true,
colorlinks=true,
linkcolor=blue,
anchorcolor=blue}

\usepackage[noend]{algpseudocode}

\usepackage{algorithmicx,algorithm}

\newtheorem{thm}{Theorem}[section]
\newtheorem{cor}{Corollary}[section]
\newtheorem{lm}{Lemma}[section]
\newtheorem{defi}{Definition}[section]
\newtheorem{prop}{Proposition}[section]

\begin{document}
\graphicspath{{./Figs/}}

\title{\bf  Hypergraph Analysis Based on a Compatible Tensor Product Structure}

\author{Jiaqi Gu \footnote{E-mail: 20110180020@fudan.edu.cn. School of Mathematical Sciences, Fudan University, Shanghai, 200433, P. R. of China. This author is supported by the National Natural Science Foundation of China under grant 12271108 and Innovation Program of Shanghai Municipal Education Commission. } 
\quad    Shenghao Feng \footnote{E-mail: 21110180034@fudan.edu.cn. School of Mathematical Sciences, Fudan University, Shanghai, 200433, P. R. of China. This author is supported by  the National Natural Science Foundation of China under grant 12271108 and
	Shanghai Municipal
	Science and Technology Commission under grant 23WZ2501400.}
\quad	 Yimin Wei \footnote{Corresponding author (Y. Wei). E-mail: ymwei@fudan.edu.cn and yimin.wei@gmail.com. School of Mathematical Sciences and Shanghai Key Laboratory of Contemporary
	 	Applied Mathematics, Fudan University, Shanghai, 200433, P. R. China. This author is supported by  the National Natural Science Foundation of China under grant 12271108 and Innovation Program of Shanghai Municipal Education Commission.
 }}


\maketitle

 \begin{abstract} 
 We propose a tensor product structure that is compatible with the hypergraph structure. We define the algebraic connectivity of the $(m+1)$-uniform hypergraph in this product, and prove the relationship with the vertex connectivity. We introduce some connectivity optimization problem into the hypergraph, and solve them with the algebraic connectivity. We introduce the Laplacian eigenmap algorithm to the hypergraph under our tensor product.

 	\bigskip

    
    
 \end{abstract}

\section{Introduction}
Hypergraph is a generalization of the graph. Many complex systems can be modeled by the hypergraph, and the hypergraph is widely used in many areas. Compared with graphs, the hypergraph generalizes the concept of the edge, enabling a hyperedge to contain more than two vertices. Therefore hypergraph models have much more flexibility than graph models. Some relationships that involves more than 2 people, such as the co-author relationship and the contacts among students \cite{benson2018simplicial}, cannot be directly described by a graph, but can be naturally modeled by a hypergraph. Such flexibility increases the difficulty when analyzing the properties of the hypergraph. 
For the recent years, scientists from computer science \cite{2016Directed,1993Directed}, complex networks \cite{2008Hypernetworks,2016Generalized} and mathematics \cite{Keqin1996Spectra,2012Algebraic,2014Regular,Sun2016moore} have been focusing on the hypergraph. The intention and tools may differ as the background of the study differs, leading to diverse structures and analysis.

The connectivity is a major concern in the theoretical analysis of both the graph and the hypergraph. In graph theory, the connectivity of a connected graph can be further measured by the edge connectivity and the vertex connectivity \cite{Bapat2014graphs}. The connectivity problems have many variations, like minimum cut or splitting, maximum flow and so on. In complex network, connectivity means the robustness. The robustness describes the ability to resist attacks that can be modeled by removing a vertex or an edge. Therefore the robustness can be described by the edge connectivity and the vertex connectivity. However, computing the edge connectivity or the vertex connectivity is quite difficult. Also, the edge connectivity and the vertex connectivity are discrete, therefore are too rough when analyzing some small perturbations of the graph. Fiedler propose the algebraic connectivity of a graph and proved the relationship among the algebraic connectivity, the vertex connectivity and the edge connectivity \cite{1973Algebraic} and fix the two shortcomings. In complex network, there are many studies using the algebraic connectivity as a measure of the robustness of networks \cite{ghosh2006growing,rewiring2008,bisection2010,sydney2013optimizing,wei2014algebraic,cheung2021improving,mackay2021finding, wei2013algebraic}. 
Some results about the convergence in the consensus problems also involves the algebraic connectivity \cite{qi2018consensus,kar2007consensus,kar2008sensor}. 
There are also some other applications of the algebraic connectivity in the complex system. According to \cite{nagarajan2015maximizing}, there are different applications of the algebraic connectivity in vehicle localization, unmanned aerial vehicles rigid formations and UAVs network synthesis with data congestion. Maximizing the algebraic connectivity under some constraints is a common problem in network science. There are also theoretical studies on maximizing the algebraic connectivity over some certain families of graphs \cite{maximizing2017,kolokolnikov2015maximizing}.

In the hypergraph analysis, there are different approaches. There are some matrix-based approaches \cite{2016Generalized,1993Directed}, which generalize the adjacency matrix to the hypergraph using the incident matrix. Graph reduction techniques \cite{HypergraphCuts,lawler1973cutsets,agarwal2006higher} split a hyperedge into several edges, reducing a hypergraph into a graph with multiple edges. Another method is the tensor analysis. The same as the adjacency matrix, an $m$-uniform hypergraph with $n$ vertices is naturally corresponding to an $m$th order $n$-dimensional tensor. Hu and Qi \cite{2012Algebraic} represented a $2m$-uniform hypergraph as a tensor and utilized the H-eigenvalues and the Z-eigenvalues \cite{Qi2017tensor,Wan2022SpectraOW,Xie2016Spectral} to define the algebraic connectivity. There are also some studies based on the Einstein product of tensors \cite{chen2019multilinear,MultilinearControl}. 
In this paper we propose a new product between tensors to fit the structure of the hypergraph. This leads to the following contributions.

\begin{enumerate}
    \item We propose a new product between tensors. This product is compatible with the structure of the hypergraph. Furthermore, this product also works between a vector and a tensor. We propose the quadratic form, the eigenvalue and eigenvector under this product. 
    \item We define the algebraic connectivity of the hypergraph under this product and prove the relationship between the algebraic connectivity and the vertex connectivity. We analyze the influence on the algebraic connectivity when adding a hyperedge based on the Fielder vector. We generalize two classic problems in complex network to the hypergraph.
    \item We make some numerical examples for the algebraic connectivity to illustrate our results. We make some numerical examples for the two hypergraph problems.
    \item We propose the hypergraph Laplacian eigenmap algorithm under out tensor product structure, and make some numerical examples showing its advantages.
\end{enumerate}

In some cases, our analysis seems to be equivalent to the clique reduction technique \cite{HypergraphCuts,lawler1973cutsets}. However, an original hypergraph and its tensor description keeps much more information than the reduced graph and its adjacency matrix.


\section{Preliminaries}
In this section, we provide basic notations throughout this paper.
\subsection{Hypergraph}
    The same as graphs, a hypergraph $\mathcal{G}=(V,\mathcal{E})$ consists of two parts: the set of vertices $V$ and the set of hyperedges $\mathcal{E}$. A hyperedge $E_i \subseteq V$ consists of several vertices. A hypergraph is $m$-uniform if all its hyperedges contain the same number of vertices, i.e., $\vert E_i\vert=m$ for all $E_i\in\mathcal{E}$. Specially, a graph can be regarded as a $2$-uniform hypergraph. For an undirected hypergraph, each hyperedge is an unordered set. For a directed hypergraph, its hyperedge $E_i=(T_i,H_i)$ is further divided into two parts: the head $H_i$ and the tail $T_i$ \cite{2016Directed,1993Directed,2008Hypernetworks}. We follow the definition in \cite{ausiello2001directed,Xie2016Spectral} that assumes $\vert H\vert=1$, or the definition of the forward arc in \cite{2016Directed,1993Directed,Cambini1997FlowsOH}. A directed hyperedge is an ordered pair $E=(T,h)$, where the unordered set $T\subset V$ is the tail and $h\in V\backslash T$ is the head. In the following part, we will use the hypergraph to refer to an undirected hypergraph, and will use the directed hypergraph to clarify the directed property.
    
    A path in graph theory can be generalized to the hypergraph as a hyperpath \cite{Li2016Extremal}. We use $[n]$ as an abbreviation for the set $\{1,2,\dots,n\}$. For the undirected hypergraph, a hyperpath of length $l$ is a sequence of hyperedges and vertices $<E_1,v_1,E_2,v_2,\dots,v_{l-1},E_l>$ such that 
    \begin{equation*}
        v_i\in E_i\cap E_{i+1},\ \forall i\in[l-1],
    \end{equation*}
    and $v_i\neq v_{i+1}$ for all $i\in[l-2]$.
    For the directed hypergraph, a hyperpath $<E_1,v_1,E_2,\dots,v_{l-1},E_l>$ of length $l$ satisfies
    \begin{equation*}
        v_i=h_i {\ \rm and\ } v_i\in T_{i+1},\ \forall i\in[l-1].
    \end{equation*}
    As hyperedges forbid duplicate vertices, the property $v_i\neq v_{i+1}$ still holds for the directed hypergraph as $v_i=h_i\in T_{i+1}$ and $v_{i+1}=h_{i+1}\notin T_{i+1}$.
    
    Based on this we can define the product of two $(m+1)$-uniform hypergraphs. In graph theory, the matrix product of two graphs is defined based on their adjacency matrices \cite[Definition 17.3]{Sudhakara2023ProductGM}. 
    Suppose $G_1=(V,E_1)$ and $G_2=(V,E_2)$ are defined on the same vertex set $V$ and have their adjacency matrices $A_1$ and $A_2$. Then the product matrix $A_3=A_1A_2$ also corresponds to a graph $G_3=(V,E_3)$ that allows multiple edges and self-loops. $G_3$ is exactly the matrix product of $G_1$ and $G_2$, and every edge $e=(v_i,v_j)$ in $E_3$ corresponds to a path $[(v_i,v_k),(v_k,v_j)]$ of length $2$. In the path, the first edge $(v_i,v_k)$ comes from $G_1$ and the second $(v_k,v_j)$ comes from $G_2$. Or saying, $(v_i,v_k)\in E_1$ and $(v_k,v_j)\in E_2$. Specially, if $G_1=G_2$, then we have a conclusion \cite[Lemma 8.1.2]{Godsil2001Algebraic}.
    \begin{thm}
        Suppose $G$ is a directed graph and $A$ is its adjacency matrix. Then the number of walks from $v_i$ to $v_j$ with length $r$ is $(A^r)_{ij}$. 
    \end{thm}
    
    The product of two $(m+1)$-uniform hypergraphs is similar. 
    The product of $\mathcal{G}_1=(V_1,\mathcal{E}_1)$ and $\mathcal{G}_2=(V_2,\mathcal{E}_2)$ is an $(m+1)$-uniform hypergraph $\mathcal{G}=(V,\mathcal{E})$. For a directed hypergraph, $E=(T,h)\in \mathcal{E}$ if and only if there exists a hyperpath of length two $<E_1=(T,h_1),h_1,E_2=(T_2,h)>$ such that $E_i\in\mathcal{E}_i$ for $ i=1,2$. For an undirected hypergraph, we only need to split every hyperedge $E=\{v_{i_1},\dots,v_{i_{m+1}}\}$ into $(m+1)$ directed hyperedges $\{(E\backslash\{v_{i_j}\},v_{i_j})\}_{j=1}^{m+1}$, and the other part is the same. 
    
    The connectivity of the hypergraph can be defined via the hyperpath. For the undirected graph, two vertices $v_0,v_l$ is connected if there exists a hyperpath 
    $<E_1,v_1,\dots,v_{l-1},E_l>,$
    such that $v_0\in E_1$ and $v_l\in E_l$.
    Whether $v_0=v_1$ is not essential. If $v_0=v_1$, we can remove $E_1$ and $v_1$ from the hyperpath. This also works on $v_{l-1}=v_{l}$. Therefore without loss of generality we may assume $v_0\neq v_1$ and $v_{l-1}\neq v_{l}$ to be compatible with the definition of the hyperpath. Further we say an undirected hypergraph to be connected if every two vertices are connected. 
    For the directed hypergraph, a vertex $v_0$ has access to a vertex $v_l$ if there exists a hyperpath 
    $ <E_1,v_1,\dots,v_{l-1},E_l>,$
    such that $v_0\in T_1$ and $v_l=h_l$. We use $v_i \to v_j$ to represent this. The connectivity can be further distinguished. A directed hypergraph is said to be
    \begin{itemize}
        \item strongly connected, if every vertex $v_i$ has access to every vertex $v_j$;
        \item one-way connected, if for every two vertices $v_i$ and $v_j$, $v_i\to v_j$ or $v_j\to v_i$;
        \item weak connected, if the base hypergraph (switching every directed hyperedge $E=(T,h)$ into an undirected hyperedge $T\cup\{h\}$) is connected.
    \end{itemize}
    
    For a connected hypergraph, the connectivity can be further measured by the vertex connectivity and the edge connectivity. A cutset $\widehat{V}\subseteq V$ of an undirected hypergraph $\mathcal{G}$ is a subset of $V$ such that if we remove the vertices in $\widehat{V}$ and all their incident hyperedges from $\mathcal{G}$, it will not be connected. The vertex connectivity $v(\mathcal{G})$ is the minimum size of the cutset. Let $\mathcal{G}(V\backslash\widehat{V})$ be the hypergraph induced by $V\backslash\widehat{V}$, or saying, removing the vertices in $\widehat{V}$ and all their incident hyperedges from $\mathcal{G}$. Then the vertex connectivity can be represent as
    \begin{equation*}
        v(\mathcal{G})=\min_{\widehat{V}\subset V} \{\vert \widehat{V}\vert:\  \mathcal{G}(V\backslash\widehat{V}) {\rm\ is\ not\ connected}\}.
    \end{equation*}
    For the directed hypergraph, removing a cutset will break the weak connectivity.
    
    The hyperedge connectivity $e(\mathcal{G})$ is similar. $e(\mathcal{G})$ is the minimum number of hyperedges we need to remove to break the (weak) connectivity of a (directed) hypergraph.
    \begin{equation*}
        e(\mathcal{G})=\min_{\widehat{\mathcal{E}}\subset\mathcal{E}} \{\vert\widehat{\mathcal{E}}\vert:\ \widehat{\mathcal{G}}=(V,\mathcal{E}\backslash\widehat{\mathcal{E}}) {\rm\ is\ not\ connected}\}.
    \end{equation*}
    The vertex connectivity and the edge connectivity have some variations. The min-cut problem asks to split a graph (hypergraph) into two parts while minimizing the cost of edges (hyperedges) between these two parts, which is a variation of the edge connectivity. In the robust analysis of a complex network, the attacks on a vertex (hyperedge) will disable it, therefore leads to the vertex (hyperedge) connectivity. 
  
\subsection{Tensor and Existing Tensor Products}
    \label{section2.2}
    Tensor is a high order generalization of the matrix, and a matrix is called a second order tensor. A tensor \cite{Che2020theory,Qi2017tensor,Wei2016theory} is a multidimensional array. The order is the number of its indices.  A tensor is denoted by calligraphic letters $\mathcal{A},\mathcal{B}$ and so on. 
    The set of all the $m$th order $n$-dimensional real tensors is denoted as $T_{m,n}$. 
    
    There are different structures for tensor analysis. Based on the high order homogeneous polynomial, a tensor-vector product between $\mathcal{A}\in T_{m+1,n}$ and $x\in\mathbb{R}^n$ can be defined \cite{Qi2017tensor} as
    \begin{equation*}
        \mathcal{A}x^m=(\sum_{i_1,\dots,i_m=1}^n a_{i_1,\dots,i_m,j}x_{i_1}\cdots x_{i_m})_{j=1}^m \in\mathbb{R}^n.
    \end{equation*}
    Based on this the Z-eigenvalue and the Z-eigenvector $(\lambda,x)$ is defined \cite{Qi2017tensor} as
    \begin{equation*}
        \mathcal{A}x^{m}=\lambda x,\ \Vert x\Vert_{\rm 2}=1.
    \end{equation*}
    The H-eigenvalue and the H-eigenvector $(\lambda,x)$ is introduced \cite{Qi2017tensor} as 
    \begin{equation*}
        \mathcal{A}x^{m}=\lambda x^m=\lambda (x_1^m,\dots,x_n^m)^\top.
    \end{equation*}
    The same as the quadratic form of matrices, a homogeneous polynomial is defined as
    \begin{equation*}
        \mathcal{A}x^{m+1}=x^\top(\mathcal{A}x^m)=\sum_{i_1,\dots,i_{m+1}=1}^n a_{i_1,\dots,i_m,i_{m+1}}x_{i_1}\cdots x_{i_{m+1}}.
    \end{equation*}
    Only when $(m+1)$ is even, it can be positive or negative definite, as a polynomial of odd degree can never be positive or negative definite.
    
    For the $2m$th order $n$-dimensional tensors, there is another structure, called the Einstein product. It is first proposed by Nobel laureate Einstein in \cite{albert1916foundation} and there are some studies based on this structure \cite{2013Solving,Miao2020fourth}. The Einstein product of  tensors $\mathcal{A} \in \mathbb{R}^{n_1\times \cdots \times n_k \times p_{1} \times  \cdots \times p_m}$ and $\mathcal{B} \in \mathbb{R}^{p_{1}\times \cdots \times p_m \times q_{1} \times \cdots  \times q_l}$ is a tensor in $\mathbb{R}^{n_1\times \cdots \times n_k \times q_{1}\times \cdots \times q_l}$ such that
    \begin{equation*}
        (\mathcal{A} *_m \mathcal{B})_{i_1,\dots,i_k, j_{1}, \dots,j_l} =\sum_{t_1,t_2,\dots,t_m} a_{i_1,\dots, i_k, t_1,\dots,t_m} b_{t_1,\dots,t_m, j_{1}, \dots,j_l}.
    \end{equation*}
    The ring $(T_{2m,n},+,*_m)$ is isomorphism to the matrix ring $(\mathbb{R}^{n^m\times n^m},+,\cdot)$ \cite{2013Solving}. The identity tensor $\mathcal{I}$ satisfies $\mathcal{I}_{i_1,\dots,i_m,i_1,\dots,i_m}=1$ for $1\leq i_1,\dots,i_m\leq n$ and all the other elements are $0$. The generalized eigenvalue and the eigentensor \cite{Wang2022GeneralizedEF} $(\lambda,\mathcal{X})$ is defined as 
    \begin{equation*}
        \mathcal{A}*_m \mathcal{X}=\lambda\mathcal{B}*_{m}\mathcal{X},
    \end{equation*}
    where $\mathcal{X}\in T_{m,n}$. When $\mathcal{B}=\mathcal{I}$, the generalized  eigenvalue problem is just the classic eigenvalue problem.

\subsection{Laplacian Matrix of Graph, Algebraic Connectivity and Laplacian Eigenmap}
    In graph theory, for a graph $G=(V,E)$ with $n$ vertices, its adjacency matrix is $A\in\mathbb{R}^{n\times n}$, such that $a_{ij}=1$ if $(i,j)\in E$ and $a_{ij}=0$ otherwise. If we allow multiple edges, then $a_{ij}$ is the number of edges $(i,j)$ in $E$. Its degree diagonal matrix is $D={\rm diag}(\{d_i\}_{i=1}^n)$. $d_i$ is the degree of the vertex $v_i$, satisfying $(d_1,\dots,d_n)^\top=A{\bf 1}$, where $\bf 1$ is the all-one column vector.
    For a directed graph, $d_i$ is the outdegree. The Laplacian matrix is defined as $L=D-A$ for both the directed and undirected graph.
    
    The Laplacian matrix processes many interesting properties. The Laplacian matrix is diagonally dominant and all its diagonal elements are non-negative. For an undirected graph $G$, its Laplacian matrix is symmetric, therefore positive semi-definite. By its definition we have $L{\bf 1}=0$, therefore $\bf 1$ is an eigenvector corresponding to the smallest eigenvalue $0$. The second smallest eigenvalue $a(G)$ of $L$ is the algebraic connectivity of $G$, which is proposed by Fiedler in \cite{1973Algebraic}. It can also be expressed by the Courant–Fischer theorem \cite{Bapat2014graphs,1973Algebraic}. Let ${\bf E}={\rm span}\{\bf 1\}$. We have \cite{1993Directed,wu2005algebraic}
    \begin{equation}
       a(G)=\min_{\|x \|_2=1, \atop x  \in {\bf E}^{\perp} } x^\top Lx.
       \label{agdef}
    \end{equation}
    For a directed graph, its algebraic connectivity is defined as \eqref{agdef} in \cite{wu2005algebraic}. However, as $L$ is not symmetric, $a(G)$ may be negative, as illustrated by the disconnected graph in \cite{wu2005algebraic}. For a real symmetric matrix $A$ of order $n$, let the eigenvalues of $A$ be arranged as:
    $ \lambda_1(A) \leq \lambda_2(A) \leq \cdots \leq \lambda_n(A) $. It is well known that \cite[Lemma 3]{wu2005algebraic}
    \begin{equation*}
            \lambda_1\left(\frac{1}{2}\left(L+L^\top\right)\right) \leq a(G) \leq \lambda_2\left(\frac{1}{2}\left(L+L^\top \right)\right).
    \end{equation*}
    
    The algebraic connectivity is widely used as a measure of the robustness of complex networks \cite{ghosh2006growing,rewiring2008,bisection2010,sydney2013optimizing,wei2014algebraic,cheung2021improving,mackay2021finding}. When adding an edge, the algebraic connectivity helps to measure the increment of the robustness of the network. In the consensus problem of the graph, the convergence speed of a linear system
    \begin{equation}
        \dot{x}(t)=-Lx(t)
        \label{gcp}
    \end{equation}
    can be measured by the algebraic connectivity, the second smallest eigenvalue of $L$ \cite{qi2018consensus}.

    The same as the graph, a tensor can be derived from a (directed) uniform hypergraph, called the adjacency tensor. For an undirected $(m+1)$-uniform hypergraph $\mathcal{G}=(V,\mathcal{E})$ with $|V|=n$, its adjacency tensor \cite{Cooper2012spectra} is an $(m+1)$th order $n$-dimensional tensor $\mathcal{A}=\left(a_{i_1,\dots,i_m,i_{m+1}}\right)$, satisfying.
    \begin{equation}
        a_{i_1,\dots,i_m,i_{m+1}}=
        \left\{
        \begin{array}{cc}
            \frac{1}{m!} & {\rm If\ } \{i_1,\dots,i_{m+1}\}\in\mathcal{E}\\
            0 & {\rm otherwise.}
        \end{array}
        \right.
        \label{atdef}
    \end{equation}
    The Laplacian tensor can be defined as $\mathcal{L}=\mathcal{D}-\mathcal{A}$, where $\mathcal{D}$ is an $(m+1)$th order $n$-dimensional diagonal tensor of which diagonal elements are the degrees of vertices. In different structures, the coefficients $\frac{1}{m!}$ in \eqref{atdef} may differ. The definition of the degree may also differ.
    Hu and Qi defined the algebraic connectivity for the hypergraph using the $H$-eigenvalue and $Z$-eigenvalue \cite{2012Algebraic} of its Laplacian tensor. There are also studies about spectrum \cite{2014Regular,Xie2016Spectral} under the same structure.
    
    The Laplacian eigenmap \cite{2003Laplacian,2001Laplacian} is widely used in dimensionality reduction in data science. Given a dataset $(x_1,\dots,x_n)\in\mathbb{R}^{p\times n}$, a dimensionality reduction algorithm aims to project the dataset onto a space with lower dimension $q<p$, while keeping the distance between data to some extent. For each $x_i$, the Laplacian eigenmap algorithm connects a vertex $v_i$ with its neighbors, for example, $k$ nearest neighbors or neighbors of which distance is less than a given threshold $\epsilon$. Then a weighed graph is generated. The Laplacian eigenmap projects the dataset onto the subspace spanned by the eigenvectors corresponding to the $q$ smallest non-zero eigenvalues of the normalized Laplacian matrix. There are also some matrix-based hypergraph Laplacian eigenmap surveys.

\subsection{Graph Reduction Techniques for Hypergraph}
    Graph reduction techniques \cite{HypergraphCuts,lawler1973cutsets,agarwal2006higher} reduce a hypergraph into a related graph, and therefore classic graph methods that allow multiple edges can be applied. Clique reduction \cite{HypergraphCuts,lawler1973cutsets} splits a hyperedge $E=\{v_{i_1},\dots,v_{i_{m+1}}\}$ into a clique $r(E)=\{\{v_{i_j},v_{i_k}\}:\ 1\leq j<k\leq m+1\}$. For a directed hypergraph, the reduction of $(T=\{v_{i_1},\dots,v_{i_{m}}\},v_{i_{m+1}})$ is $r(E)=\{(v_{i_j},v_{i_{m+1}}):\ j\in[m]\}$. 
    For a hypergraph $\mathcal{G}=(V,\mathcal{E})$, its reduced graph $r(\mathcal{G})=(V,r(\mathcal{E}))$ satisfies
    \begin{equation*}
        r(\mathcal{E})=\bigsqcup_{E\in\mathcal{E}} r(E),
    \end{equation*}
    where $\bigsqcup$ is the disjoint union as we allows multiple edges in $r(\mathcal{G})$.
    
    Graph reduction can help solving many problems. For example, a hypergraph cut $V=S\cup S^{\rm c}$ with penalty function 
    \begin{equation}
        f(S)=\sum_{E\in\mathcal{E}} \vert E\cap S\vert\vert E\cap S^{\rm c}\vert g(S,S^{\rm c})
        \label{timespenalty}
    \end{equation}
    can be naturally transformed into the graph cut problem of $r(\mathcal{G})$ by the clique reduction. $g(S,S^{\rm c})$ is a normalized function, such as $\frac{1}{\vert S\vert}+\frac{1}{\vert S^{\rm c}\vert}$ and $\frac{1}{{\rm vol}(S)}+\frac{1}{{\rm vol}( S^{\rm c})}$ where ${\rm vol}(S)=\sum_{v\in S} d_v$. We can also consider the isoperimetric number of a hypergraph $\mathcal{G}$  \cite{li2017analytic}
    \begin{equation*}
        i(\mathcal{G})=\min_{S\subset V}\left\{\frac{\vert \{E\in\mathcal{E}:E\cap S\neq \emptyset,\ E\cap S^{\rm c} \neq\emptyset\}\vert}{\min\{\vert S\vert,\vert S^{\rm c}\vert\}}\right\},
    \end{equation*}
    which corresponds to the all-or-nothing cut function in \cite{veldt2022hypergraph}.
    It can be modeled by the circle reduction, which reduce a hyperedge 
    $\{v_{i_1},\dots,v_{i_{m+1}}\}$ into 
    \begin{equation*}
        \{\{v_{i_j},\ v_{i_{j+1}}\}:\ j=1,\dots,m+1\},
    \end{equation*}
    where $v_{i_{m+2}}=v_{i_1}$. After the reduction, we have
    \begin{equation*}
        i(r(\mathcal{G}))=2i(\mathcal{G}).
    \end{equation*}

\section{A Compatible Tensor Product Structure}
In this section, we propose a tensor product structure that is compatible with the hypergraph structure.
\subsection{A New Tensor Product in $T_{m+1,n}$}
    To fit the structure of the hypergraph better, we propose a new product between two tensors $\mathcal{A}\in T_{m+1,n}$ and $\mathcal{B}\in T_{k+1,n}$. 
    \begin{defi}
        The product between two tensors $\mathcal{A}\in T_{m+1,n}$ and $\mathcal{B}\in T_{k+1,n}$ is a new tensor $\mathcal{C}\in T_{m+1,n}$, satisfying
        \begin{equation*}
            c_{i_1,\dots,i_m,i_{m+1}}=\sum_{j_1,j_2,\dots,j_k,t=1}^n a_{i_1,\dots,i_m,t}b_{j_1,\dots,j_k,i_{m+1}}\delta(t\in\{j_1,\dots,j_k\}),
        \end{equation*}
        in which $\delta(p)$ is like the Kronecker function. If the statement $p$ is true, then $\delta(p)=1$. Otherwise $\delta(p)=0$.
        \label{proddefi}
    \end{defi}
    This product shows great compatibility with the matrix product. When $m=k=1$, the product is just the same as the standard matrix product. When $m=0$, it leads to a vector-tensor product $x^\top \mathcal{B}$ and if $k=1$ this time, it is the same as the standard vector-matrix product. 

    The associative law holds for this product. For $\mathcal{A},\mathcal{B}$ and $\mathcal{C}$ we can assert $(\mathcal{A}*\mathcal{B})*\mathcal{C}=\mathcal{A}*(\mathcal{B}*\mathcal{C})$. This product is not communicative.
    
    We then focus the product between tensors in $T_{m+1,n}$, i.e., $k=m$. We define the diagonal tensor $\mathcal{D}={\rm diag}\left((a_i)_{i=1}^n\right)$ as $d_{i,\dots,i}=a_i$ for $i\in[n]$ and $d_{i_1,\dots,i_{m+1}}=0$ otherwise.
    The right identity $\mathcal{I}\in T_{m+1,n}$ is $\mathcal{I}={\rm diag}\left((1)_{i=1}^n\right)$, satisfying
    \begin{equation*}
        \mathcal{A}*\mathcal{I}=\mathcal{A},\ \forall\mathcal{A}\in T_{m+1,n}.
    \end{equation*}
    
    We move on to the vector-tensor product, denoting it as $x^\top\mathcal{A}$ for $x\in\mathbb{R}^n$ and $\mathcal{A}\in T_{m+1,n}$. By the Definition \ref{proddefi} $x^\top\mathcal{A}\in\mathbb{R}^{1\times n}$ and 
    \begin{equation*}
        (x^\top\mathcal{A})_j=\sum_{i_1,\dots,i_m,i} x_i a_{i_1,\dots,i_m,j}\delta(i\in\{i_1,\dots,i_m\})
        =\sum_{i_1,\dots,i_m}  a_{i_1,\dots,i_m,j}(x_{i_1}+\cdots+x_{i_m}).
    \end{equation*}
    Based on this we can define the eigenvalue and eigenvector under this product. The eigenvalue and eigenvector $(\lambda,x)$ satisfies
    \begin{equation*}
        x^\top \mathcal{A}=\lambda x^\top,
    \end{equation*}
    and can be solved by a linear equation
    \begin{equation*}
        x^\top (\mathcal{A}-\lambda\mathcal{I})={\bf 0}.
    \end{equation*}
    The quadratic form $(x^\top \mathcal{A}) x\in\mathbb{R}$ is
    \begin{align*}
        (x^\top \mathcal{A}) x
        &=\sum_{i_1,\dots,i_m,i_{m+1},t} x_t a_{i_1,\dots,i_m,i_{m+1}}\delta(t\in\{i_1,\dots,i_m\})x_{i_{m+1}}\\
        &=\sum_{i_1,\dots,i_m,i_{m+1}} a_{i_1,\dots,i_m,i_{m+1}}x_{i_{m+1}}(x_{i_1}+\cdots+x_{i_m}).
    \end{align*}
    
    For $x^\top\in\mathbb{R}^{1\times n}$, $x^\top\mathcal{A}$ is a linear transform on the row vector space $\mathbb{R}^{1\times n}$. Therefore it can be represented by a matrix $A\in\mathbb{R}^{n\times n}$, denoting as $A=\phi(\mathcal{A})$. $A$ can be identified by directly computing $e_k^\top \mathcal{A}$, which leads to the $k$-th column of $A$.
    \begin{equation}
        \begin{aligned}
            (e_k^\top \mathcal{A})_i&=(0,\dots,0,1,0,\dots,0)\mathcal{A}
            =\sum_{j_1,\dots,j_m,t}e_{kt}a_{j_1,\dots,j_m,i}\delta(t\in\{j_1,\dots,j_m\})\\
            &=\sum_{j_1,\dots,j_m}a_{j_1,\dots,j_m,i}\delta(k\in\{j_1,\dots,j_m\}).
        \end{aligned}
        \label{eqcal}
    \end{equation}
    As a result, the eigenvalue problem and the extremum of the quadratic form can be solved.
    
    
    Subtensor under this product is not a cube as usual. Consider a linear equation $x^\top \mathcal{A}=b$. If we focus on the reaction of $\{x_i\}_{i\in I}$ and $\{b_j\}_{j\in J}$ where $I$ and $J$ are index sets, this will lead to an irregular subset of $\mathcal{A}$ 
    \begin{equation*}
        \mathcal{A}[I,J]=\{l_{i_1,\dots,i_m,j}\}|\{i_1,\dots,i_m\}\cap I\neq \emptyset,j\in J\}.
    \end{equation*}
    The influence of $J$ is the same as the common case of submatrix. However, the shape influenced by $I$ is irregular. When $\mathcal{A}$ is $3$rd-order, the location of elements in the slice is like a union of several crosses. We will still use the expression $x_I^\top L[I,J]$.

\subsection{Adjacency Tensor, Laplacian Tensor and its Properties}
    We retain the definition of the adjacency tensor in \eqref{atdef}. For the directed hypergraph, the condition $\{i_1,\dots,i_{m+1}\}\in\mathcal{E}$ will be $(\{i_1,\dots,i_{m}\},i_{m+1})\in\mathcal{E}$.  
    We can easily identify that the adjacency tensor is permutation invariant. For an undirected hypergraph and a permutation 
    $$\sigma:[m+1]\to[m+1],$$
    we have $a_{i_{\sigma(1)},\dots,i_{\sigma(m+1)}}=a_{i_1,\dots,i_{m+1}}$ because they correspond to the same hyperedge. For a directed hypergraph and a permutation $\sigma:\ [m]\to[m]$, we have $a_{i_{\sigma(1)},\dots,i_{\sigma(m)},i_{m+1}}=a_{i_1,\dots,i_{m+1}}$. 

    Our product is compatible with the structure of the hypergraph. Suppose $\mathcal{A}_i$ is the adjacency tensor of $\mathcal{G}_i$ for $i=1,2$. Then $\mathcal{A}=\mathcal{A}_1*\mathcal{A}_2$ also corresponds to a hypergraph. Its hyperedge corresponds to a hyperpath of length $2$, the first hyperedge comes from $\mathcal{G}_1$ and the second comes from $\mathcal{G}_2$. Moreover, if $i_1\neq i_{m+1}$ holds for all $a_{i_1,\dots,i_{m+1}}\neq 0$, then $\mathcal{A}$ is exactly an adjacency tensor and $a_{i_1,\dots,i_{m+1}}=\frac{k}{m!}$ means there are exactly $k$ such hyperpaths. $i_1= i_{m+1}$ corresponds to the self-loop situation in graph theory.

    The degree of vertices $(d_i)_{i=1}^n$ can be defined as 
    \begin{equation}
        (d_i)_{i=1}^n={\bf 1}^\top\mathcal{A}=\left(m \sum_{i_1,\dots,i_m}^n a_{i_1,\dots,i_m,i}\right)_{i=1}^n.
        \label{defideg}
    \end{equation}
    As the definition of the Laplacian tensor is $\mathcal{L}={\rm diag}(\{d_i\}_{i=1}^n)-\mathcal{A}$, we have ${\bf 1}^\top \mathcal{L}={\bf 0}$.
    For a directed hypergraph, $(d_i)_{i=1}^n$ is the indegree of vertices. 
    
    Although in the undirected hypergraph the definition of the degree of a vertex \eqref{defideg} leads to 
    \begin{equation*}
        {\rm deg}(v_i)=m\vert\{E\in\mathcal{E}:v_i\in E\}\vert,
    \end{equation*}
    we retain the coefficient $m$ so that in a directed hypergraph the total outdegree can match the total indegree. In a directed hypergraph, the definition \eqref{defideg} leads to the indegree of a vertex: 
    \begin{equation*}
        {\rm deg}_{\rm in}(v_i)=m\vert\{E=(T,h)\in\mathcal{E}:v_i=h\}\vert.
    \end{equation*}
    As the tail of a hyperedge contains $m$ vertices and the head contains only one, we have 
    \begin{equation*}
    \begin{aligned}
        \sum_i {\rm deg}_{\rm in}(v_i)&=\sum_i m\vert\{E=(T,h)\in\mathcal{E}:v_i=h\}\vert\\
        &=\sum_i\vert\{E=(T,h)\in\mathcal{E}:v_i\in T\}\vert=\sum_i {\rm deg}_{\rm out}(v_i),
    \end{aligned}
    \end{equation*}
    where the outdegree of vertex $v_j$ can be counted by 
    \begin{equation*}
        {\rm deg}_{\rm out}(v_i)=\vert\{E=(T,h)\in\mathcal{E}:v_i\in T\}\vert=\sum_{i_1,\dots,i_m,i_{m+1}} a_{i_1,\dots,i_m,i_{m+1}}\delta(i\in\{i_1,\dots,i_m\}).
    \end{equation*}

    For the undirected graph, its adjacency matrix and Laplacian matrix is symmetric. There is also a similar result for the undirected hypergraph.
    
    \begin{lm}
        The linear representation $\phi(\mathcal{A})$ of the adjacency tensor $\mathcal{A}$ is symmetric if the hypergraph $\mathcal{G}$ is undirected. This also holds for the Laplacian tensor.
        \label{lmSymMax}
    \end{lm}
    
    \begin{proof}
        For an $(m+1)$-uniform undirected hypergraph, a hyperedge $E=\{v_{i_1},\dots,v_{i_{m+1}}\}$ will lead to $(m+1)!$ non-zero elements in $\mathcal{A}$: $a_{\sigma({i_1}),\dots,\sigma({i_{m+1}})}=\frac{1}{m!}$ where $\sigma$ is an arbitrary permutation over $\{1,\dots,m+1\}$. Then by \eqref{eqcal} we have
        \begin{equation*}
            \begin{aligned}
                \phi(\mathcal{A})_{ki}&=\sum_{j_1,\dots,j_m}a_{j_1,\dots,j_m,i}\delta(k\in\{j_1,\dots,j_m\})\\
                &=m\sum_{j_2,\dots,j_m}a_{k,j_2,\dots,j_m,i}.\ \ \ {\rm Use\ the\ permutation.}
            \end{aligned}
        \end{equation*}
        
        Then use the permutation invariant property over all the hyperedges that contains vertices $v_k$ and $v_i$, we have
        \begin{equation*}
            \phi(\mathcal{A})_{ki}=m\sum_{j_2,\dots,j_m}a_{k,j_2,\dots,j_m,i}
            =m\sum_{j_2,\dots,j_m}a_{i,j_2,\dots,j_m,k}=\phi(\mathcal{A})_{ik}.
        \end{equation*}
    \end{proof}

    In the following parts, we will use $\widehat{i}$ as an abbreviation for the ordered indices $(i_1,\dots,i_{m+1})$ or $(i_1,\dots,i_m)$ if the dimension of $\widehat{i}$ is not confusing.

\subsection{Comparison between Different Structure}    
    There are different methods to analyze the properties of the hypergraph. The tensor analysis based on the Einstein product is not compatible with the hypergraph structure. When taking the Einstein product, we need the hypergraph to be even order uniform. Suppose we consider a $2m$-uniform hypergraph with $n$ vertices in the view of Einstein product. Two adjacent hyperedge in a hyperpath need to have exactly $m$ same vertices. If we consider the consensus problem \eqref{gcp}, then $\mathcal{X}(t)\in T_{m,n}$ and $\mathcal{X}_{\widehat{i}}(t)$ does not correspond to a vertex $v_i$, but to an ordered set of $m$ vertices $(v_i)_{i\in\widehat{i}}$. Meanwhile, the degree is defined for the $m$-vertices set rather than the vertex itself. This is not common in graph theory and complex network. Moreover, if we consider the splitting or clustering tasks, splitting the $m$-vertices set is meaningless and we only concern the clustering of these vertices themselves. The H-eigenvalue and the Z-eigenvalue structure partially resolve these problems. For \eqref{gcp} we have $x(t)\in\mathbb{R}^n$ and each $x_i(t)$ corresponds to a vertex $v_i$. However, there are still some problems. The positive semi-definite property still asks the order to be even. Also, both the H-eigenvalue and the Z-eigenvalue are high order polynomials. Computing them is much more difficult. Our tensor product structure fixes these problems. We can summarize its advantages as below.  Compared with the Einstein product, our product has the following advantages.
    \begin{enumerate}
        \item The Einstein product is only for the even order uniform hypergraph. Our product has no such limit.
        \item In a hyperpath, two adjacent hyperedges only need to have at least one common vertex under our product structure, rather than exactly $m$ common vertices for an $2m$-uniform hyperedge under the Einstein product.
        \item In the models like \eqref{gcp}, each vertex is equipped with a state, rather than each $m$-vertices set. This is more natural and explicable in network theory. The clustering and partition can be done on the vertices, rather than on the $m$-vertices set.
        \item Much less computational cost.
    \end{enumerate}
    Compared with other tensor products, our product has the following advantages.
    \begin{enumerate}
        \item The positive semi-definite property of the Laplacian tensor needs the hypergraph to be even order uniform. Our product does not have such limit.
        \item Much less computational cost.
    \end{enumerate}
    
    In some cases, our structure is consistent with the clique reduction. For $\mathcal{G}$ and its adjacency tensor $\mathcal{A}$, the adjacency matrix of its reduced graph $r(\mathcal{G})$ in the clique reduction is exactly $\phi(\mathcal{A})$, the matrix representation of the linear mapping of $\mathcal{A}$. This also happens on the Laplacian tensor and will be shown in the following section. This consistency makes our structure the same as a matrix-based analysis of the reduced graph in some cases. For example, consider the min-cut problem \cite{HypergraphCuts} with penalty function \eqref{timespenalty}.
    As 
    \begin{equation*}
        f(S)={\bf 1}^\top_S \mathcal{L}{\bf 1}_S={\bf 1}^\top_S L_{r(\mathcal{G})}{\bf 1}_S
    \end{equation*}
    holds for the indicator vector ${\bf 1}_S$, a tensor-based analysis in our structure is the same as the matrix-based analysis for the reduced graph $r(\mathcal{G})$ if the matrix analysis method allows multiple edges. For example, the consensus problem of the hypergraph under this product has form
    \begin{equation*}
        \dot{x}(t)=-x(t)^\top\mathcal{L},
    \end{equation*}
    which also corresponds to the continuous-time homogeneous polynomial dynamical system \cite[Eq. (6)]{Chen2023Explicit}.
    It is the same as a graph consensus problem \eqref{gcp} after the clique reduction.
    Some results related to the adjacency matrix also have this consistency, such as the Hoffman theorem \cite[Thm 3.23]{Bapat2014graphs}.
    
    However, our tensor product structure and the tensor-based analysis still have some unique advantages. Some graph methods that do not support multiple edges can not be applied to the reduced graph. For example, if we allow multiple edges, the relationship between the algebraic connectivity and the vertex connectivity in \cite{1993Directed,wu2005algebraic} no longer holds, and the new relationship will be established in the following section using our tensor-based analysis. The hypergraph and the adjacency tensor also carry more information than the reduced graph and its adjacency matrix. Moreover, the hypergraph is a unique structure. Suppose a hyperedge $\{v_1,\dots,v_{m+1}\}$ is split into a clique $\{\{v_i,v_j\}:1\leq i<j\leq m+1\}$ and now we remove the vertex $v_1$. In the original hypergraph, the whole hyperedge is removed. In its adjacency tensor, we remove the elements of which indices contain $1$, leading to an $(m+1)$th order $(n-1)$-dimensional tensor. In the reduced graph, only $\{\{v_1,v_i\}:1< i\leq m+1\}$ is removed and $\{\{v_i,v_j\}:2\leq i<j\leq m+1\}$ are retained. In the adjacency matrix, it is an $(n-1)\times(n-1)$ matrix containing $\{\{v_i,v_j\}:2\leq i<j\leq m+1\}$. According to this analysis, we have such conclusion.
    \begin{prop}
        For a hypergraph $\mathcal{G}$ and its reduced graph $r(\mathcal{G})$, their vertex connectivity satisfies
        \begin{equation*}
            v(\mathcal{G})\leq v(r(\mathcal{G})).
        \end{equation*}
    \end{prop}

    
\section{Algebraic Connectivity of Hypergraph}
    In this section, we define the algebraic connectivity of the hypergraph via our tensor product structure and prove some properties, making it possible to serve as a measure of robustness. We also generalize two classic problem in complex network to the hypergraph.
\subsection{Algebraic Connectivity for the Undirected Hypergraph}
    In this and the next subsection we focus on the undirected hypergraph. 
    Based on Lemma \ref{lmSymMax} we can defined the algebraic connectivity of the hypergraph by the quadratic form.
    \begin{defi}
        Let ${\bf 1}$ be the all ones vector. The algebraic connectivity of an $(m+1)$-uniform hypergraph $\mathcal{G}$ is defined by quadratic form
        \begin{equation*}
            a(\mathcal{G})=\min_{x\perp {\bf 1},\Vert x\Vert_2=1} x^\top \mathcal{L} x.
        \end{equation*}
        \label{defiACU}
    \end{defi}
    
    We have a direct conclusion.
    \begin{thm}
        Algebraic connectivity $a(\mathcal{G})\geq0$. $a(\mathcal{G})=0$ if and only if the hypergraph is not connected.
        \label{agpsd}
    \end{thm}
    \begin{proof}
         This can be proved by analyzing the influence of a certain hyperedge. Suppose $\{\mathcal{G}_E\}_{E\in\mathcal{E}}$ is the set of the sub-hypergraph satisfying $\mathcal{G}_E=(V,\{E\})$. Let $\{\mathcal{L}_E\}_{E\in\mathcal{E}}$ be the corresponding Laplacian tensors, then $\mathcal{L}=\sum_{E\in\mathcal{E}} \mathcal{L}_E$ and thus $x^\top \mathcal{L}x=\sum_{E\in\mathcal{E}} x^\top\mathcal{L}_E x$. Based on this we can analyzing the influence of a certain hyperedge $x^\top\mathcal{L}_E x$. Without loss of generality we may assume $E=\{v_1,\dots,v_{m+1}\}$. Then
        \begin{equation*}
        \begin{aligned}
            x^\top \mathcal{L}_E x&=x^\top \mathcal{D}_E x-x^\top \mathcal{A}_E x=\sum_{j=1}^n ({\bf 1}^\top \mathcal{A}_E)_j x_j^2 -\sum_{j=1}^n (x^\top \mathcal{A}_{E})_j x_j\\
            &=\sum_{j=1}^n (x_j^2\sum_{i,i_1,\dots,i_m} a_{i_1,\dots,i_m,j}\delta(i\in\{i_1,\dots,i_m\}) )
            - \sum_{j=1}^n (x_i x_j\sum_{i,i_1,\dots,i_m} a_{i_1,\dots,i_m,j}\delta(i\in\{i_1,\dots,i_m\}) )\\
            &=m\sum_{j=1}^{m+1} x_j^2 - \sum_{i,j=1,i\neq j}^{m+1} x_i x_j=\sum_{1\leq i<j\leq m+1} (x_i-x_j)^2.
        \end{aligned}
        \end{equation*}
        It is positive semi-definite. In fact, it is the same as a quadratic form of the Laplacian matrix of a complete graph $K_{m+1}$. Therefore $x^\top \mathcal{L}_E x\geq 0$. Then 
        \begin{equation*}
            x^\top \mathcal{L}x=\sum_{E\in\mathcal{E}}\sum_{i,j\in E,i<j}(x_i-x_j)^2\geq 0.
            \label{sumXLX}
        \end{equation*}
        Suppose $\{v_1,\dots,v_k\}$ is a connected component of $\mathcal{G}$. If $x^\top \mathcal{L} x=0$ and $x_1=c$, then $x_j=c$ for all $j\in [k]$. As a result, all the vertices in a connected component share a same value.
        If $\mathcal{G}$ is connected, then $x^\top \mathcal{L} x=0$ implies $x=c{\bf 1}$. Therefore $a(\mathcal{G})>0$ holds for the connected hypergraph. 
        
    \end{proof}
    
    We actually prove $\mathcal{L}$ is positive semi-definite and the algebraic multiplicity of eigenvalue $0$ equals to the number of its connected components. In this proof, a hypergraph is equivalent to a graph generated by the clique reduction. However, the advantages of the hypergraph over the graph reduction will be seen in Lemma \ref{MainLM}.
    
    As the Laplacian tensor of an undirected hypergraph is positive semi-definite and ${\bf 1}^\top\mathcal{L}=0$, there is another representation of the algebraic connectivity of the undirected hypergraph.
    \begin{prop}
        The algebraic connectivity $a(\mathcal{G})$ is the second smallest eigenvalue of its Laplacian tensor.
    \end{prop}
    The eigenvector corresponding to the algebraic connectivity is called the Fiedler vector. 

    We further move onto the relationship between the algebraic connectivity and the vertex connectivity.

    \begin{lm}
        If $\mathcal{G}_i=(V,\mathcal{E}_i)$ for $i=1,2$ and $\mathcal{E}_1\cap \mathcal{E}_2=\emptyset$, then $a(\mathcal{G}_1)+ a(\mathcal{G}_2)\leq a( \mathcal{G}_1\cup\mathcal{G}_2)$
    \end{lm}
    \begin{proof}
        \begin{equation*}
            a(\mathcal{G}_1\cup\mathcal{G}_2)=\min_{x\perp {\bf 1},\Vert x\Vert_2=1} x^\top(\mathcal{L}_1+\mathcal{L}_2)x\geq \min_{x\perp {\bf 1},\Vert x\Vert_2=1} x^\top\mathcal{L}_1x+\min_{x\perp {\bf 1},\Vert x\Vert_2=1} x^\top\mathcal{L}_2x
            =a(\mathcal{G}_1)+a(\mathcal{G}_2).
        \end{equation*}
    \end{proof}
    As the algebraic connectivity of an unconnected hypergraph is $0$, we have
    \begin{cor}
        If $\mathcal{G}_1\subseteq \mathcal{G}_2$, then $a(\mathcal{G}_1)\leq a(\mathcal{G}_2)$.
    \end{cor}

    Sparsity is an important hypothesis for the hypergraph. There are much more possible hyperedges in a hypergraph. A simple graph with $n$ vertices can only have $\frac{1}{2}n(n-1)$ edges at most. However, an $(m+1)$-uniform hypergraph can contain at most $\binom{n}{m+1}$ hyperedges. When $m\ll n$, this is almost an exponential rate of $m$. Specially, in a graph, a vertex can be incident to at most $(n-1)$ edges; In an $(m+1)$-uniform hypergraph, a vertex can be incident to $\binom{n-1}{m}$ hyperedges. This will boost up the eigenvalues of $\mathcal{A}$ and $\mathcal{L}$. For example, in a complete $5$-uniform hypergraph with $10$ vertices, each vertex is incident to $126$ hyperedges and the smallest non-zero eigenvalue of the Laplacian tensor is $560$. Therefore a sparsity hypothesis is needed. The sparsity $s$ of a hypergraph is defined as 
    \begin{equation*}
        s=\max_{v_i,v_j}\vert\{E\in\mathcal{E}:\ v_i,\ v_j\in E\}\vert.
    \end{equation*}
    We further prove that when removing a vertex and all its incident hyperedges, the upper bound for the loss of the algebraic connectivity is related to $s$.

    \begin{lm}
    \label{MainLM}
        Let $\mathcal{G}$ be a hypergraph with sparsity $s$, and $\mathcal{G}_1$ is derived by removing an arbitrary vertex and all its incident hyperedges, then 
        \begin{equation*}
            a(\mathcal{G}_1)\geq a(\mathcal{G})-(2m-1)s.
        \end{equation*}
    \end{lm}
    \begin{proof}
        Suppose we remove the last vertex $v_n$. Let $\mathcal{L}_1$ be the Laplacian of $\mathcal{G}_1$ and $y$ be the Fiedler vector. 
        
        The Laplacian tensor $\mathcal{L}$ of $\mathcal{G}$ can be divided into four parts: $\mathcal{L}[V_{n-1},V_{n-1}]$, $\mathcal{L}[V_{n-1},v_n]$, $\mathcal{L}[v_n,V_{n-1}]$ and $\mathcal{L}[v_n,v_n]=l_{n,\dots,n}={\rm deg}(v_n)$. In the first part $\mathcal{L}[V_{n-1},V_{n-1}]$, the diagonal elements vary due to the loss of degrees, denoting as $\mathcal{D}_1$. The elements corresponding to the hyperedges that contain $v_n$ also vary. The other elements are the same as $\mathcal{L}$. Thus we have
        \begin{equation*}
        \begin{aligned}
            &(y^\top\ 0)\mathcal{L}(y^\top\ 0)^\top\\
            =&y^\top (\mathcal{L}_1+\mathcal{D}_1)y+\sum_{i,j=1}^{n-1}\sum_{i_1,\dots,i_m=1}^n y_i y_j l_{i_1,\dots,i_m,j}\delta(i\in\{i_1,\dots,i_m\})\delta(n\in\{i_1,\dots,i_m,j\})\\
            =&a(\mathcal{G}_1)+y^\top \mathcal{D}_1 y+\sum_{i,j=1}^{n-1}\sum_{i_1,\dots,i_m=1}^n y_i y_j l_{i_1,\dots,i_m,j}\delta(i\in\{i_1,\dots,i_m\})\delta(n\in\{i_1,\dots,i_m\}).
        \end{aligned}
        \end{equation*}
        
        
        For $y^\top\mathcal{D}_1 y$ we have 
        \begin{equation*}
        \begin{aligned}
            y^\top \mathcal{D}_1 y
            &=\sum_{j=1}^{n-1}y_j^2 \sum_{i_1,\dots,i_m=1}^n ma_{i_1,\dots,i_m,j} \delta(n\in\{i_1,\dots,i_m\})
            =\sum_{j=1}^{n-1}y_j^2\sum_{i_2,\dots,i_m=1}^n m^2 a_{n,i_2,\dots,i_m,j}\\
            &=\sum_{j=1}^{n-1}y_j^2\sum_{i_2,\dots,i_m=1}^n m^2 a_{j,i_2,\dots,i_m,n}
            =\sum_{j=1}^{n-1}y_j^2\sum_{i_1,\dots,i_m=1}^n m a_{j,i_2,\dots,i_m,n}\delta(j\in\{i_1,\dots,i_m\})\\
            &=\sum_{i_1,\dots,i_m=1}^n a_{i_1,\dots,i_m,n}\sum_{j=1}^m m y_{i_j}^2 .
        \end{aligned}
        \end{equation*}
        For the last summation term we have
        \begin{equation*}
        \begin{aligned}
            &\sum_{i,j=1}^{n-1}\sum_{i_1,\dots,i_m=1}^n y_i y_j l_{i_1,\dots,i_m,j}\delta(i\in\{i_1,\dots,i_m\})\delta(n\in\{i_1,\dots,i_m\})\\
            =&\sum_{i,j=1}^{n-1}\sum_{i_3,\dots,i_m=1}^n m(m-1)y_i y_j l_{n,i,i_3,\dots,i_m,j}
            = \sum_{i,j=1}^{n-1} m(m-1)\left(\sum_{i_3,\dots,i_m=1}^n l_{i,j,i_3,\dots,i_m,n}\right)y_i y_j \\
            =& \sum_{i,j=1\atop i\neq j}^{n-1} \left(\sum_{i_1,\dots,i_m=1}^n l_{i_1,i_2,\dots,i_m,n}\right)\delta(i,j\in\{i_1,\dots,i_m\})y_i y_j 
            = \sum_{i_1,\dots,i_m=1}^n l_{i_1,i_2,\dots,i_m,n} \sum_{i,j\in\{i_1,\dots,i_m\}\atop i\neq j} y_i y_j.
        \end{aligned}
        \end{equation*}
        Therefore
        \begin{equation}
        \label{MainAC}
        \begin{aligned}
        &y^\top \mathcal{D}_1 y+\sum_{i,j=1}^{n-1}\sum_{i_1,\dots,i_m=1}^n y_i y_j l_{i_1,\dots,i_m,j}\delta(i\in\{i_1,\dots,i_m\})\delta(n\in\{i_1,\dots,i_m\})\\
        =&\sum_{i_1,\dots,i_m}^n a_{i_1,\dots,i_m,n}\left(\sum_{j=1}^m m y_{i_j}^2 -\sum_{i,j\in\{i_1,\dots,i_m\}\atop i\neq j} y_i y_j\right)
        =\sum_{E\in\mathcal{E}\atop v_n\in E}\sum_{1\leq i<j< n\atop v_i,v_j\in E} (y_i-y_j)^2+\sum_{E\in\mathcal{E}\atop v_n\in E}\sum_{v_i\in E\atop i<n} y_i^2.
        \end{aligned}
        \end{equation}
        The first term in \eqref{MainAC} is a quadratic form. As the sparsity is $s$, for all $1\leq i<n$ we have
        \begin{equation*}
            \vert\{E\in\mathcal{E}:v_i,v_n\in E\}\vert\leq s.
        \end{equation*}
        Therefore
        \begin{equation*}
            \left\vert\{(v_j,\ E):\ v_i,v_n,v_j\in E,\ E\in\mathcal{E}\}\right\vert\leq (m-1)s.
        \end{equation*}
        Then
        \begin{equation*}
        \begin{aligned}
            \sum_{E\in\mathcal{E}\atop v_n\in E}\sum_{1\leq i<j< n\atop v_i,v_j\in E} (y_i-y_j)^2
            &\leq \sum_{E\in\mathcal{E}\atop v_n\in E}\sum_{1\leq i<j< n\atop v_i,v_j\in E} 2(y_i^2+y_j^2)\\
            &\leq 2\sum_{i=1}^{n-1} y_i^2\left\vert\{(v_j,\ E):\ v_i,v_n,v_j\in E,\ E\in\mathcal{E}\}\right\vert\leq 2(m-1)s.
        \end{aligned}
        \end{equation*}
        
        For the second term, we have
        \begin{equation*}
            \sum_{E\in\mathcal{E}\atop v_n\in E}\sum_{v_i\in E\atop i<n} y_i^2
            =\sum_{i=1}^{n-1} \vert\{E\in\mathcal{E}:v_i,v_n\in E\}\vert y_i^2
            \leq \sum_{i=1}^{n-1} s y_i^2\leq s.
        \end{equation*}
        Therefore we can assert
        \begin{equation*}
            y^\top \mathcal{D}_1 y+\sum_{i,j=1}^{n-1}\sum_{i_1,\dots,i_m=1}^n y_i y_j l_{i_1,\dots,i_m,j}\delta(i\in\{i_1,\dots,i_m\})\delta(n\in\{i_1,\dots,i_m\})\leq (2m-1)s.
        \end{equation*}

        In conclusion, we have
        \begin{equation*}
            a(\mathcal{G})\leq (y^\top\ 0)\mathcal{L}(y^\top\ 0)^\top \leq a(\mathcal{G}_1)+ (2m-1)s,
            \label{Cdefi}
        \end{equation*}
    \end{proof}
    
    When $m=1$, we have $(2m-1)=1$ and $s=1$. The result is consistent with the result in graph theory \cite{1973Algebraic}.
    
    
    By induction we have a direct corollary.
    
    \begin{cor}
    \label{corCK}
        Let $\mathcal{G}$ be a hypergraph with sparsity $s$. 
        Let $\mathcal{H}$ be the hypergraph that removes $k$ vertices and their incident hyperedges from $\mathcal{G}$, then 
        \begin{equation*}
            a(\mathcal{H})\geq a(\mathcal{G})-(2m-1)sk.
        \end{equation*}

    \end{cor}

    Lemma \ref{lmsym} is a symmetric situation of Corollary \ref{corCK}.
    \begin{lm}
    \label{lmsym}
        Let $\mathcal{G}_i=(V_i,\mathcal{E}_i),\ i=1,2$ be a partition of the hypergraph $\mathcal{G}$, then
        \begin{equation*}
            a(\mathcal{G})\leq a(\mathcal{G}_1)+(2m-1)s\vert V_2\vert,a(\mathcal{G}_2)+(2m-1)s\vert V_1\vert).
        \end{equation*}
        
    \end{lm}
    
    From Lemma \ref{lmsym} and \ref{agpsd} we can construct the relationship between the algebraic connectivity and the vertex connectivity.
    \begin{thm}
        Suppose a hypergraph $\mathcal{G}$ has sparsity $s$. Then
        \begin{equation*}
            a(\mathcal{G})\leq (2m-1)s v(\mathcal{G}).
        \end{equation*}
    \end{thm}

\subsection{Algebraic Connectivity Maximization}
    Algebraic connectivity maximization is a common problem in network science. There are studies in different scenarios, leading to different constraints. In this section we generalize two classic scenarios to the hypergraph. The first one is about the hyperedge rewiring or adding problem. The graph situation is studied in \cite{bisection2010,ghosh2006growing,rewiring2008,sydney2013optimizing}. In both rewiring and adding, we focus on the influence of a certain hyperedge. 
    In another aspect, when we want to increase the robustness of a hypergraph complex network, we always consider the algebraic connectivity maximization via the hyperedge adding or rewiring method. Therefore we consider adding a hyperedge $E_0\notin\mathcal{E}$ to $\mathcal{G}$, denoting as $\mathcal{G}'=\mathcal{G}+E_0=(V,\mathcal{E}\cup\{E_0\})$. Let $q_2$ be the Fiedler vector of $\mathcal{G}$. By considering $q_2^\top \mathcal{L}_{\mathcal{G}'}q_2$ we have an upper bound for $a(\mathcal{G}')$.
    
    \begin{prop}
    \label{Propedgeub}
    \begin{equation*}
        a(\mathcal{G}')\leq a(\mathcal{G})+q_2^\top \mathcal{L}_{E_0} q_2.
    \end{equation*} 
    \end{prop}

    Then we move onto the lower bound. A analysis for the $3$-uniform hypergraph is established. Suppose $\lambda_1\leq\cdots\leq\lambda_n$ is the eigenvalue of $\mathcal{L}$. By the interlacing theorem we can assert $\lambda_2\leq a(\mathcal{G}+E_0)\leq\lambda_4$. However, we cannot assert $a(\mathcal{G}+E_0)\leq\lambda_3$ like the result in graph theory, because there is a simple counter-example by considering $\lambda_2=\lambda_3$. If $a(\mathcal{G}+E_0)\leq\lambda_3$, we have such result.
        
    
    \begin{thm}
        Let $\mathcal{G}$ be a $3$-uniform hypergraph and $\mathcal{G}+E_0$ be the hypergraph by adding a hyperedge $E_0$. $\lambda_1\leq\cdots\leq \lambda_n$ be the eigenvalue of $\mathcal{L}_{\mathcal{G}}$ and $q_2$ be the Fiedler vector. If $ a(\mathcal{G}+E_0)\leq\lambda_3$, then there exists a lower bound for $a(\mathcal{G}+E_0)$ in the interval $(\lambda_2,\lambda_3)$ that monotonically increases with $q_2^\top \mathcal{L}_{E_0} q_2$.
    \label{Thmedgelb}
    \end{thm}
    \begin{proof}
    In this proof we always consider the quadratic form of $\mathcal{L}$, therefore we can consider the matrix representation $\phi(\mathcal{L})$. 
    Let $\phi(\mathcal{L})=Q\Lambda Q^\top$ be the eigenvalue decomposition of $\phi(\mathcal{L})$ and $Q=[q_1|\dots|q_n]$ be the normalized eigenvectors. $q_1=\frac{1}{\sqrt{n}}{\bf 1}$ and $q_2$ is the Fiedler vector. Let $\phi(\mathcal{L}_{E_0})$ be the matrix representation of the Laplacian tensor of $E_0$ and has eigenvalue decomposition $\phi(\mathcal{L}_{E_0})=3(u_1 u_1^\top+u_2 u_2^\top)$. 
    We have $(Q^\top u_i)_j=q_j^\top u_i$ for $j\in [n]$. Further, we have
    \begin{equation}
    \begin{aligned}
        \det(\phi(\mathcal{L})+\phi(\mathcal{L}_{E_0})-\lambda I)=&\det(\Lambda+3\sum_{i=1,2}(Q^\top u_i)(Q^\top u_i)^\top-\lambda I)\\
        =&\det(\Lambda-\lambda I)\det(I+3(\Lambda-\lambda I)^{-1}\sum_{i=1,2}(Q^\top u_i)(Q^\top u_i)^\top).
    \end{aligned}
    \label{templabel1}
    \end{equation}
    For the last term in \eqref{templabel1}, according to \cite[Thm 2.3]{wu2012relationship} we have
    \begin{equation}
    \begin{aligned}
        f(\lambda):=&\det(I+3(\Lambda-\lambda I)^{-1}\sum_{i=1}^2(Q^\top u_i)(Q^\top u_i)^\top)\\
        =&1+3\sum_{i=1}^n \frac{\sum_{j=1,2}(q_i^\top u_j)^2}{\lambda_i-\lambda}+
        9\left(\left(\sum_{i=1}^n \frac{(q_i^\top u_1)^2}{\lambda_i-\lambda}\right)\left(\sum_{i=1}^n\frac{(q_i^\top u_2)^2}{\lambda_i-\lambda} \right)
        -\left(\sum_{i=1}^n \frac{(q_i^\top u_1) (q_i^\top u_2)}{\lambda_i-\lambda}\right)^2 \right).\\
    \end{aligned}
    \label{expadd3}
    \end{equation}
    As $q_i=\frac{1}{\sqrt{n}} {\bf 1}$ and $u_i\perp {\bf 1}$, all the summation in \eqref{expadd3} actually starts from $i=2$.
    
    As $\lambda_2<\lambda<\lambda_3$, then $f(\lambda)<0$ leads to $\lambda\leq a(\mathcal{G}')$ and therefore can serve as a lower bound. 
    Therefore we need to maximize $\lambda$ while keeping $f(\lambda)\leq 0$.
    As $\lambda_2-\lambda<0$ and $\lambda_i-\lambda>0$ for $i=3,\dots,n$, from a direct computation we have
    \begin{equation*}
    \begin{aligned}
        \left(\sum_{i=2}^n \frac{(q_i^\top u_1)^2}{\lambda_i-\lambda}\right)\left(\sum_{i=2}^n\frac{(q_i^\top u_2)^2}{\lambda_i-\lambda} \right)
        =&\left(\sum_{i=3}^n\frac{(q_i^\top u_1)^2}{\lambda_i-\lambda}\right)\left(\sum_{i=3}^n\frac{(q_i^\top u_2)}{\lambda_i-\lambda}\right)+\frac{(q_2^\top u_1)^2(q_2^\top u_2)^2}{(\lambda_2-\lambda_i)^2}\\
        &+
        \left(\frac{(q_2^\top u_1)^2}{\lambda_2-\lambda}\right)\left(\sum_{i=3}^n\frac{(q_i^\top u_2)^2}{\lambda_i-\lambda}\right)
        +\left(\frac{(q_2^\top u_2)^2}{\lambda_2-\lambda}\right)\left(\sum_{i=3}^n\frac{(q_i^\top u_1)^2}{\lambda_i-\lambda}\right),
    \end{aligned}        
    \end{equation*}
    and
    \begin{equation*}
    \begin{aligned}
        \left(\sum_{i=2}^n \frac{(q_i^\top u_1) (q_i^\top u_2)}{\lambda_i-\lambda}\right)^2
        &=\frac{(q_2^\top u_1)^2(q_2^\top u_2)^2}{(\lambda_2-\lambda_i)^2}+\left(\sum_{i=3}^n \frac{(q_i^\top u_1) (q_i^\top u_2)}{\lambda_i-\lambda}\right)^2\\
        &+2\left(\frac{(q_2^\top u_1) (q_2^\top u_2)}{\lambda_2-\lambda}\right)\left(\sum_{i=3}^n \frac{(q_i^\top u_1) (q_i^\top u_2)}{\lambda_i-\lambda}\right).
    \end{aligned}
    \end{equation*}
    Therefore the last term in \eqref{expadd3} can be computed as
    \begin{equation}
    \begin{aligned}
        &\left(\sum_{i=2}^n \frac{(q_i^\top u_1)^2}{\lambda_i-\lambda}\right)\left(\sum_{i=2}^n\frac{(q_i^\top u_2)^2}{\lambda_i-\lambda} \right)
        -\left(\sum_{i=2}^n \frac{(q_i^\top u_1) (q_i^\top u_2)}{\lambda_i-\lambda}\right)^2 \\
        =&\left(\sum_{i=3}^n \frac{(q_i^\top u_1)^2}{\lambda_i-\lambda}\right)\left(\sum_{i=3}^n\frac{(q_i^\top u_2)^2}{\lambda_i-\lambda} \right)
        -\left(\sum_{i=3}^n \frac{(q_i^\top u_1) (q_i^\top u_2)}{\lambda_i-\lambda}\right)^2 
        -\left(\frac{(q_2^\top u_1)^2}{\lambda-\lambda_2}\right)\left(\sum_{i=3}^n\frac{(q_i^\top u_2)^2}{\lambda_i-\lambda}\right)\\
        &-\left(\frac{(q_2^\top u_2)^2}{\lambda-\lambda_2}\right)\left(\sum_{i=3}^n\frac{(q_i^\top u_1)^2}{\lambda_i-\lambda}\right)
        +2\left(\frac{(q_2^\top u_1) (q_2^\top u_2)}{\lambda-\lambda_2}\right)\left(\sum_{i=3}^n \frac{(q_i^\top u_1) (q_i^\top u_2)}{\lambda_i-\lambda}\right) .       
    \end{aligned}
    \label{expall1}
    \end{equation}
    According to the mean value inequality, the sum of the last three terms in \eqref{expall1} is negative. 
    As $Q$ is normalized orthogonal, $u_t=\sum_{i=1}^n (q_i^\top u_t)q_i$ and therefore $\sum_{i=1}^n (q_i^\top u_t)^2=\Vert u\Vert^2=1$ for $t=1,2$.
    For $3((q_2^\top u_1)^2+(q_2^\top u_2)^2)$ we have
    \begin{equation*}
        3\sum_{i=1,2}(q_2^\top u_i)^2=3\sum_{i=1,2}(q_2^\top u_i)(u_i^\top q_2)=3\sum_{t=1,2}q_2^\top(u_i u_i^\top)q_2=q_2^\top\left(3\sum_{t=1,2}u_i u_i^\top\right)q_2=q_2^\top L_{E_0} q_2.
    \end{equation*}
    Therefore $f(\lambda)$ have a upper bound
    \begin{equation}
    \begin{aligned}
        f(\lambda)&\leq 1+3\sum_{i=2}^n \frac{\sum_{t=1,2}(q_i^\top u_t)^2}{\lambda_i-\lambda}+
        9\left(\left(\sum_{i=3}^n \frac{(q_i^\top u_1)^2}{\lambda_i-\lambda}\right)\left(\sum_{i=3}^n\frac{(q_i^\top u_2)^2}{\lambda_i-\lambda} \right)
        -\left(\sum_{i=3}^n \frac{(q_i^\top u_1) (q_i^\top u_2)}{\lambda_i-\lambda}\right)^2 \right)\\
        &\leq 1+\frac{q_2^\top L_{E_0} q_2}{\lambda_2-\lambda}+3\frac{\sum_{i=3}^n (q_i^\top u_1)^2+\sum_{i=3}^n (q_i^\top u_2)^2}{\lambda_3-\lambda}+9\left(\sum_{i=3}^n \frac{(q_i^\top u_1)^2}{\lambda_i-\lambda}\right)\left(\sum_{i=3}^n\frac{(q_i^\top u_2)^2}{\lambda_i-\lambda} \right)\\
        &\leq 1+\frac{q_2^\top L_{E_0} q_2}{\lambda_2-\lambda}+\frac{6}{\lambda_3-\lambda}+\frac{9}{(\lambda_3-\lambda)^2}.
    \end{aligned}
    \label{eqfinal}
    \end{equation}
    Use $g(\lambda)$ to represent the last line in equation \eqref{eqfinal}.
    $g(\lambda)$ monotonically increases with $\lambda$ in the open interval $(\lambda_2,\lambda_3)$. The only real root in this interval is an lower bound for $a(\mathcal{G}+E_0)$. As $\frac{q_2^\top L_{E_0} q_2}{\lambda_2-\lambda}<0$ and the other parts in $g(\lambda)$ is positive, by considering the monotonicity of $\frac{1}{\lambda_2-\lambda}$ and $\frac{1}{\lambda_3-\lambda}$ we can assert that this root monotonically increases with $q_2^\top L_{E_0} q_2$.

    \end{proof}

    The condition $\lambda_2\leq \lambda(\mathcal{G}+E_0)\leq \lambda_3$ can be determined by several ways. A simple way is using the Proposition \ref{Propedgeub}. If $\lambda_3-\lambda_2\geq q_2^\top\mathcal{L}q_2$, then we have $a(\mathcal{G}+E_0)\leq \lambda_3$. A more concise estimation can be done base on \eqref{expadd3}. When $\lambda\to\lambda_2^+$, $f(\lambda)\to-\infty$. Therefore $a(\mathcal{G}+E_0)\leq \lambda_3$ if and only if $\lim_{\lambda\to\lambda_3^-} f(\lambda)=+\infty$. When $\lambda\to\lambda_3^-$, we have
    \begin{equation*}
    \begin{aligned}
        f(\lambda)=
        &O(1)
        +\frac{q_3^\top \mathcal{L}_{E_0}q_3}{\lambda_3-\lambda}
        -9\frac{(q_2^\top u_1)^2(q_3^\top u_2)^2+(q_2^\top u_2)^2(q_3^\top u_1)^2-2(q_2^\top u_1) (q_2^\top u_2)(q_3^\top u_1) (q_3^\top u_2)}{(\lambda-\lambda_2)(\lambda_3-\lambda)}\\
        &+9\left(\frac{(q_3^\top u_1)^2}{\lambda_3-\lambda}\right)\left(\sum_{i=4}^n\frac{(q_i^\top u_2)^2}{\lambda_i-\lambda}\right)
        +9\left(\frac{(q_3^\top u_2)^2}{\lambda_3-\lambda}\right)\left(\sum_{i=4}^n\frac{(q_i^\top u_1)^2}{\lambda_i-\lambda}\right)\\
        &-18\left(\frac{(q_3^\top u_1) (q_3^\top u_2)}{\lambda_3-\lambda}\right)\left(\sum_{i=4}^n \frac{(q_i^\top u_1) (q_i^\top u_2)}{\lambda_i-\lambda}\right) \\
        \geq&O(1)
        +\frac{q_3^\top \mathcal{L}_{E_0}q_3}{\lambda_3-\lambda}
        -9\frac{\left((q_2^\top u_1)(q_3^\top u_2)-(q_2^\top u_2)(q_3^\top u_1)\right)^2}{(\lambda-\lambda_2)(\lambda_3-\lambda)}.\\
    \end{aligned}
    \end{equation*}
    $u_1$ and $u_2$ are the eigenvector of $\mathcal{L}_{E_0}$. Without loss of generality we may assume $E_0=\{1,2,3\}$ and let $u_1=\frac{1}{\sqrt{2}}(1,-1,0,{\bf 0})^\top$, $u_2=\frac{1}{\sqrt{6}}(1,1,-2,{\bf 0})^\top$. Then we have
    \begin{equation*}
        9\frac{\left((q_2^\top u_1)(q_3^\top u_2)-(q_2^\top u_2)(q_3^\top u_1)\right)^2}{(\lambda-\lambda_2)}
        =9\frac{\left(q_2^\top (u_1 u_2^\top-u_2 u_1^\top)q_3\right)^2}{(\lambda-\lambda_2)}
        =3\sqrt{3}\frac{(q_2^\top \widehat{L}q_3)^2}{\lambda-\lambda_2},
    \end{equation*}
    where
    \begin{equation*}
    \widehat{L}=
    \begin{pmatrix}
        \begin{array}{ccc}
            0 & 1 & -1 \\
            -1 & 0 & 1 \\
            1 & -1 & 0 \\
        \end{array} & {\bf O}\\
        {\bf O} & {\bf O}\\
    \end{pmatrix}.
    \end{equation*}
    If $q_3^\top \mathcal{L}_{E_0}q_3>3\sqrt{3}\frac{(q_2^\top \widehat{L}q_3)^2}{\lambda_3-\lambda_2}$, we have $f(\lambda)\to+\infty$ as $\lambda\to\lambda_3^-$, therefore $a(\mathcal{G}+E_0)\leq \lambda_3$. 

    Based on Prop \ref{Propedgeub} and Thm \ref{Thmedgelb}, $a(\mathcal{G}+E_0)$ can be both lower and upper bounded by terms that are positive related with $q_2^\top \mathcal{L}_{E_0} q_2$.
    Therefore we purpose a hyperedge adding or rewiring algorithm using $q_2^\top \mathcal{L}_{E_0} q_2$ as a metric.

    \begin{algorithm}[H]
    \caption{Hyperedge Rewiring or Adding} 
    \hspace*{0.02in} {\bf Input:} 
    A hypergraph $\mathcal{G}=(V,\mathcal{E})$ and its Laplacian tensor $\mathcal{L}$, number of hyperedge to rewiring or adding $N$.\\
    \hspace*{0.02in} {\bf Output:} 
    The hyperedge set $\mathcal{E}$ after rewiring.
    \begin{algorithmic}[1]
    
    \State $i=0$.
    \While{$i<N$} 
        \State Compute the algebraic connectivity and the Fiedler vector $x$ of $\mathcal{G}$.
        \State Find $E_{0}\in\mathcal{E}$ such that $x^\top \mathcal{L}_{E_0} x$ is minimized. If not rewiring, omit this step.
        \State Find $E_{1}\notin\mathcal{E}$ such that $x^\top \mathcal{L}_{E_1} x$ is maximized.
        \State Add $E_1$ into $\mathcal{E}$. Remove $E_0$ from $\mathcal{E}$ if rewiring.
        \State $i=i+1$.
    \EndWhile
    \State \Return $\mathcal{E}$.
    \end{algorithmic}
    \end{algorithm}

    We make some numerical examples, which are shown in Section 5.2.
    
    The second case is a generalization of the alternate randomized consensus under communication cost constraints (ARCCC) problem in \cite{kar2008sensor}. The original randomized consensus under communication cost constraints (RCCC) problem aims to maximize a nonconvex convergence measure, and the ARCCC problem uses the algebraic connectivity as a substitute, leading to a convex problem.
    The ARCCC problem can be generalized to the hypergraph ARCCC problem. Let $w_{\widehat{i}}$ be the weight of a hyperedge $\{v_{i_1},\dots,v_{i_{m+1}}\}$ and $c_{\widehat{i}}$ be its cost. Then the adjacency tensor for the weighed hypergraph satisfies $a_{\widehat{i}}=\frac{w_{\widehat{i}}}{m!}$. The upper bound for the total cost is $U$. Then the hypergraph ARCCC problem can be described as
    \begin{equation*}
    	\max_{\mathcal{L}} \lambda_2(\mathcal{L}), 
        \begin{aligned}
                        {\rm subject\ to\ } 
            \begin{cases}
             -1\leq \mathcal{L}_{\widehat{i}}\leq 0 {\rm\ for\ } \widehat{i}\neq(i_1,\dots,i_1),\\
            \mathcal{L} {\rm \ is\ permutation\ invariant,}\\ 
             \mathcal{L}_{\widehat{i}} =0 {\rm\ if\ } i_1=i_2 {\rm\ and\ } \widehat{i}\neq(i_1,\dots,i_1),\\
             {\bf 1}^\top \mathcal{L}={\bf 0},\\
            \sum_{\widehat{i}} \mathcal{L}_{\widehat{i}} c_{\widehat{i}}\leq U.
            \end{cases}
        \end{aligned}
    \end{equation*} 
    In this problem, $\mathcal{L}$ can not be directly transformed to $L=\phi(\mathcal{L})$ via the graph reduction, because there may not exist a hypergraph of which the reduced graph corresponds to $L$. However, it can be transformed to a semi-definite programming problem with inequality constraints by considering the graph reduction of each candidate hyperedge. The hypergraph ARCCC problem can be transformed as

    \begin{equation*}
    	 \max_{w_{\widehat{i}},s} s, 
        \begin{aligned}
                       {\rm subject\ to\ } 
                       \begin{cases}
                        0\leq w_{\widehat{i}}\leq 1 {\rm\ for\ } i_1<\cdots<i_{m+1},\\
            \sum_{i_1<\cdots<i_{m+1}} w_{\widehat{i}}L_{\widehat{i}}-s(I-\frac{1}{n}{\bf 1}{\bf 1}^\top)\succeq 0,\\
            \sum_{i_1<\cdots<i_{m+1}} w_{\widehat{i}} c_{\widehat{i}}\leq U,
            \end{cases}
        \end{aligned}
    \end{equation*}
    where $L_{\widehat{i}}=\phi(\mathcal{L}_{E_{\widehat{i}}})$. It is a convex problem and can be solved via the barrier function method.

\subsection{Directed Hypergraph}
    Similarly, we define the algebraic connectivity of the directed hypergraph.
    \begin{defi}
        Let ${\bf 1}$ be the all ones vector. The algebraic connectivity of an $(m+1)$-uniform directed hypergraph $\mathcal{G}$ is defined by quadratic form
        \begin{equation*}
            a(\mathcal{G})=\min_{x\perp {\bf 1},\Vert x\Vert_2=1} x^\top \mathcal{L} x.
        \end{equation*}
    \end{defi}
    
    The properties of the directed hypergraph are not so good as the undirected situation.
    The algebraic connectivity can be $0$ or negative for the strongly connected hypergraph. 
    We focus on the relationship between the algebraic connectivity and the vertex connectivity. The same as the undirected situation, we have $1_{\mathcal{C}}^\top \mathcal{L}1_{\mathcal{C}}^\top=0$ for a connected component $\mathcal{C}$ of the undirected base hypergraph of $\mathcal{G}$. Then we have such lemma.
    
    \begin{lm}
    \label{agneg}
        If a hypergraph $\mathcal{G}$ is not weak connected, we have
        \begin{equation*}
            a(\mathcal{G})\leq 0.
        \end{equation*}
        
    \end{lm}
    
    The order-preserving property still holds.

    \begin{lm}
        If $\mathcal{G}_i=(V,\mathcal{E}_i)$ for $i=1,2$ and $\mathcal{E}_1\cap \mathcal{E}_2=\emptyset$, then $a(\mathcal{G}_1)+ a(\mathcal{G}_2)\leq a( \mathcal{G}_1\cup\mathcal{G}_2)$
    \end{lm}
    The proof is the same. The sparsity hypothesis is also necessary. $s$ is defined as 
    \begin{equation*}
        s=\max_{v_i,v_j}\vert\{E\in\mathcal{E}:E=(T,v_j),v_i\in T\}\vert.
    \end{equation*}
    
    \begin{lm}
         Let $\mathcal{G}$ be a hypergraph and $\mathcal{G}_1$ is derived by removing an arbitrary vertex and all its $s$ out hyperedges and all the in hyperedges, then 
        \begin{equation*}
            a(\mathcal{G}_1)\geq a(\mathcal{G})-\frac{3}{2}ms.
        \end{equation*}
    \end{lm}
    \begin{proof}
        The main idea is the same like the proof of Lemma \ref{MainLM}. Let $\mathcal{L}_1$ be the Laplacian tensor of $\mathcal{G}_1$ and 
        \begin{equation*}
            y=\arg \min_{x\perp {\bf 1},\Vert x\Vert_2=1} x^\top \mathcal{L} x.
        \end{equation*}
        
        
        The same as the proof of Lemma \ref{MainLM}, we have

        \begin{equation}
             (y^\top\ 0)\mathcal{L}(y^\top\ 0)^\top=a(\mathcal{G}_1)+\sum_{E=(T,v_n)\in\mathcal{E}}  (y^\top\ 0)\mathcal{L}_E (y^\top\ 0)^\top+\sum_{E=(T,v_i)\in\mathcal{E}\atop v_n\in T} (y^\top\ 0)\mathcal{L}_E (y^\top\ 0)^\top.
             \label{eqD1}
        \end{equation}
        For a directed hyperedge $E=(\{v_{i_1},\dots,v_{i_m}\},v_{i_{m+1}})$, the quadratic form is
        \begin{equation*}
            x^\top \mathcal{L}_E x=\sum_{j=1}^m x_{i_{m+1}}^2-x_{i_{m+1}}x_{i_j}.
        \end{equation*}
        Therefore the first summation term in \eqref{eqD1} is $0$. For the second summation term, we have
        \begin{equation*}
        \begin{aligned}
            &\sum_{E=(T,v_i)\in\mathcal{E}\atop v_n\in T} (y^\top\ 0)\mathcal{L}_E (y^\top\ 0)^\top.
            =\sum_{E=(T,v_i)\in\mathcal{E}\atop v_n\in T} \sum_{v_j\in T} y_i^2-y_i y_j
            \leq \sum_{E=(T,v_i)\in\mathcal{E}\atop v_n\in T}\left(y_i^2+ \sum_{v_j\in T\atop j\neq n} y_i^2+\frac{1}{2}(y_i^2 y_j^2)\right)\\
            =&\sum_{i=1}^{n-1} y_i^2 \left(\frac{3m-1}{2}\vert\{E\in\mathcal{E}:E=(T,v_i)\}\vert +\frac{1}{2}\vert \{E=(T,v)\in\mathcal{E}: v_i\in T\}\vert\right)\\
            \leq &\sum_{i=1}^{n-1} \frac{3ms}{2}y_i^2=\frac{3}{2}ms
        \end{aligned}
        \end{equation*}
        
        In total, we prove
        \begin{equation*}
            a(\mathcal{G})\leq (y^\top\ 0)\mathcal{L}(y^\top\ 0)^\top\leq a(\mathcal{G}_1)+ \frac{3}{2}ms.
        \end{equation*}
    \end{proof}
    
    By induction we have a direct corollary.
    \begin{cor}
        Let $\mathcal{G}$ be a directed hypergraph with sparsity $s$. Let $\mathcal{H}$ be the hypergraph that removes $k$ vertices and all their incident hyperedges from $\mathcal{G}$, then 
        \begin{equation*}
            a(\mathcal{H})\geq a(\mathcal{G})-\frac{3}{2}msk.
        \end{equation*}
        
    \end{cor}
    
    This can derive a lemma in symmetric form.
    \begin{lm}
    \label{lmsymd}
        Let $\mathcal{G}_i=(V_i,\mathcal{E}_i),\ i=1,2$ be a partition of the hypergraph $\mathcal{G}$, then
        \begin{equation*}
            a(\mathcal{G})\leq a(\mathcal{G}_1)+\frac{3}{2}ms\vert V_2\vert,a(\mathcal{G}_2)+\frac{3}{2}ms\vert V_1\vert).
        \end{equation*}
        
    \end{lm}
    
    From Lemma \ref{lmsymd} and \ref{agneg} we can construct the relationship between the algebraic connectivity and the vertex connectivity.
    \begin{thm}
        Suppose a hypergraph $\mathcal{G}$ has sparsity $s$. Then
        \begin{equation*}
            a(\mathcal{G})\leq \frac{3}{2}ms v(\mathcal{G}).
        \end{equation*}
    \end{thm}

\section{Numeric Examples}
In this section, we test some examples to verify our theory.
\subsection{Structured hypergraph}
    We apply our structure firstly on the structured hypergraph. We compute the algebraic connectivity of the hyperring, the complete hypergraph, the star hypergraph and a multi-star hypergraph with some additional structures. 

    A $3$-uniform hyperring $\mathcal{G}=(V,\mathcal{E})$ has vertex set $\{v_i:i\in [n]\}$ and hyperedge set
    \begin{equation*}
        \mathcal{E}=\{\{v_i,v_{(i+1)//n},v_{(i+2)//n}\}:i\in [n]\},
    \end{equation*}
    where $c=a//b$ indicates $0\leq c<b$ and $a\equiv c({\rm mod}\ b)$ for $a,b,c\in\mathbb{Z}$.
    The algebraic connectivity $a(\mathcal{G})$ of a hyperring is shown in Table \ref{TabCn}.
    \begin{table}[H]
        \centerline{
        \begin{tabular}{|c|c|c|c|c|c|c|c|}
            \hline
            
            $n$ & 6&7&8&9&10&11&12 \\
            \hline
            $a(\mathcal{G})$ & 5.0 & 3.952 & 3.172 & 2.589 & 2.146 & 1.804&1.536\\
            \hline
        \end{tabular}
        }
        \caption{The algebraic connectivity of the $3$-uniform $n$-dimensional hyperring.}
        \label{TabCn}
    \end{table}

    The sparsity of a $3$-uniform hyperring is always $2$. The same as the algebraic connectivity of a ring graph, the algebraic connectivity goes down when $n$ becomes larger.

    The result for the complete $3-$uniform hypergraph with $n$ vertices is shown in Table \ref{TabKn}.
    \begin{table}[H]
    \centerline{
    \begin{tabular}{|c|c|c|c|c|c|c|c|}
        \hline
        $n$ & 6&7&8&9&10&11&12 \\
        \hline
        $a(\mathcal{G})$ & 24 & 35 & 48 & 63 & 80 & 99&120\\
        \hline
    \end{tabular}
    }
        \caption{The algebraic connectivity of the $3$-uniform $n$-dimensional complete hypergraph.}
        \label{TabKn}
    \end{table}

    The vertex connectivity for a complete $3$-uniform hypergraph with $n$ vertices is $(n-2)$ and the sparsity is $(n-2)$. We then move onto the star hypergraph. 
    The hyperedge set of a complete $3$-uniform star hypergraph with $(n+1)$ vertices is defined as
    \begin{equation*}
        \mathcal{E}=\{\{v_i,v_j,v_{n+1}\}:1\leq i<j\leq n\}.
    \end{equation*}
    Its algebraic connectivity $a(\mathcal{G})$ is shown in Table \ref{TabSn}.
    \begin{table}[H]
       \centerline{
        \begin{tabular}{|c|c|c|c|c|c|c|c|}
            \hline
            
            $n$ & 6&7&8&9&10&11&12 \\
            \hline
            $a(\mathcal{G})$ & 11.0 & 13.0 & 15.0 & 17.0 & 19.0 & 21.0&23.0\\
            \hline
        \end{tabular}
        }
        \caption{The algebraic connectivity of the complete $3$-uniform star hypergraph.}
        \label{TabSn}
    \end{table}

    The sparsity is $(n-1)$.
    The complete star hypergraphs shows great linear increase of the algebraic connectivity. As the vertex connectivity is $1$, it is actually the linear relationship between the algebraic connectivity and the sparsity. 
    However, the bound is still not close. We further apply it to a multi-star hypergraph with some additional structures. Let a $3$-uniform hypergraph $\mathcal{G}=(V,\mathcal{E})$ with $V=\{v_i:i\in[2n+1]\}$. The star part of the hyperedge set is
    \begin{equation*}
        \mathcal{E}_1=\{\{v_i,v_j,v_{2n+1}\}: 1\leq i\leq n, n+1\leq j\leq 2n\}.
    \end{equation*}
    The additional structures are 
    \begin{equation}
    \begin{aligned}
        \mathcal{E}_{21}=\{\{v_i,v_{(i+1)//n},v_{(i+2//n)}\}: 1\leq i\leq n\}\cup
        \{\{v_{n+i},v_{n+(i+1)//n},v_{n+(i+2)//n}\}: 1\leq i\leq n\},\\
        \mathcal{E}_{22}=\{\{v_i,v_{(i+2)//n},v_{(i+4//n)}\}: 1\leq i\leq n\}\cup
        \{\{v_{n+i},v_{n+(i+2)//n},v_{n+(i+4)//n}\}: 1\leq i\leq n\},\\
    \end{aligned}
    \label{ads}
    \end{equation}
    and $\mathcal{E}$ satisfies $\mathcal{E}=\mathcal{E}_1\cup\mathcal{E}_{21}\cup \mathcal{E}_{22}$. The algebraic connectivity $a(\mathcal{G})$ is shown in Table \ref{TabSan}.
    \begin{table}[H]
        \centerline{
        \begin{tabular}{|c|c|c|c|c|c|c|c|}
            \hline    
            $n$ &7&8&9&10&11&12&13 \\
            \hline
            $a(\mathcal{G})$ & 21.0 & 21.515 & 27.0 & 24.608 & 29.452 & 27.917 &31.759\\
            \hline
            $(2m-1)s v(\mathcal{G})$ & 21 & 24 & 27.0 & 30 & 33 & 36 &39\\
            \hline
        \end{tabular}
        }
        \caption{The algebraic connectivity of the $3$-uniform star hypergraph with an additional structure.}
        \label{TabSan}
    \end{table}
    The vertex connectivity $v(\mathcal{G})=1$, because $v_{2n+1}$ is the only vertex that connects the two parts.
    The sparsity is $n$, generated by $v_i$ and $v_{2n+1}$ for arbitrary $i\in[2n]$.
    It is a very close bound. For $n=7,9$, this bound is tight. 
    We then copy the structure related to $v_{2n+1}$, letting $V=\{v_i: i\in[2n+2]\}$ and for $j=1,2$
    \begin{equation*}
        \mathcal{E}_{1j}=\{\{v_i,v_j,v_{2n+j}\}: 1\leq i\leq n, n+1\leq j\leq 2n\},
    \end{equation*}
    and $\mathcal{E}_1=\cup_{j=1}^k \mathcal{E}_{1j}$. The additional structure \eqref{ads} are retained. Then the vertex connectivity is $2$. The algebraic connectivity is shown in Table \ref{TabSakn}.
    \begin{table}[H]
        \centerline{
        \begin{tabular}{|c|c|c|c|c|c|c|c|}
            \hline    
            $n$ &7&8&9&10&11&12&13 \\
            \hline
            $a(\mathcal{G})$ & 40.308 & 37.515 & 45.773 & 44.608 & 51.452 & 51.917 &57.759\\
            \hline
            $(2m-1)s v(\mathcal{G})$ & 42 & 48 & 54 & 60 & 66 & 72 &84\\
            \hline
        \end{tabular}
        }
        \caption{The algebraic connectivity of the $3$-uniform multi-star hypergraph with an additional structure.}
        \label{TabSakn}
    \end{table}
    It shows that the increase of the vertex connectivity will leads to almost linear increase of the algebraic connectivity.

\subsection{Hyperedge Rewiring and the hypergraph ARCCC Problem}

     We perform our algorithm on the Contact-Primary-School dataset and the Email-Enron dataset, which belong to the ScHoLP datasets \cite{benson2018simplicial}. These datasets contain hyperedges of different sizes, and we only use the hyperedges that contain $3$ vertices to construct a $3$-uniform hypergraph. Then we perform a hyperedge rewiring method.

    \begin{figure}[H]
        \centering\includegraphics[width=8cm]{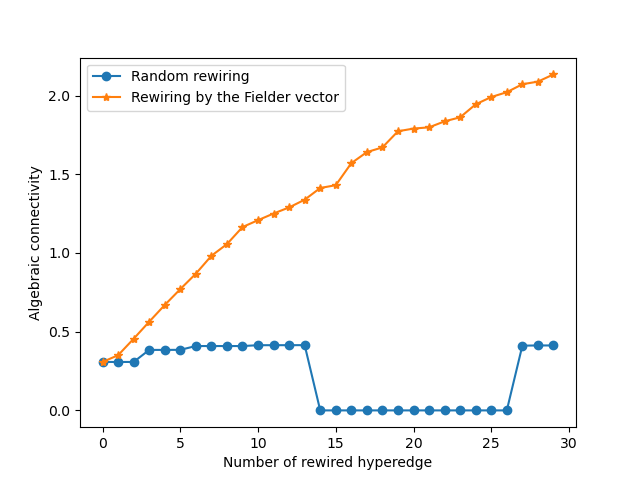}
        \centering\includegraphics[width=8cm]{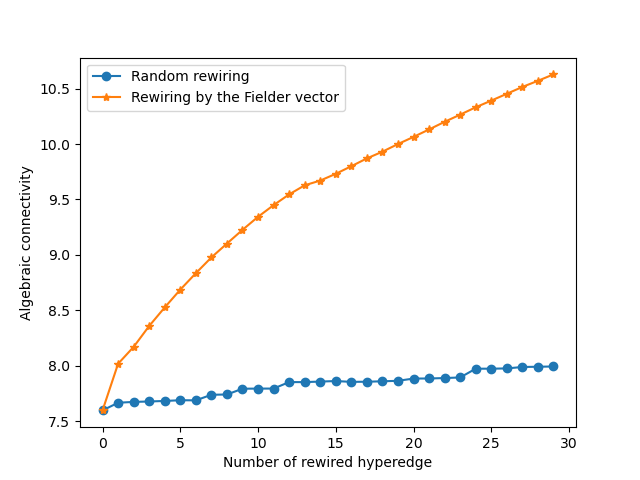}
        \caption{Results on the hyperedge rewiring problem.}       
    \label{Rewiring}  
    \end{figure}
    The left-hand side is the result of the Contact-Primary-School dataset and the right side is on the Email-Enron dataset. The random rewiring may decrease the algebraic connectivity, as shown in the first case. If the only hyperedge between two parts is removed, the algebraic connectivity will decrease to $0$ until a new hyperedge connecting these two parts is established. In comparison, our algorithm improve the algebraic connectivity efficiently. 


    For the hypergraph ARCCC problem, we firstly generate the weights in interval $[0,1]$ and the cost for each hyperedge for the Email-Enron dataset.
    In the dataset, each hyperedge repeat several times, leading to a primal weight $\Bar{w}_{\widehat{i}}\in\mathbb{N}$. The adjusted weight $w_{\widehat{i}}$ and cost $c_{\widehat{i}}$ are based on this. The more a hyperedge is repeated, the higher it weighs and the less it costs. To be specific, we normalized it weight by
    \begin{equation*}
        w_{\widehat{i}}=\sqrt[3]{\frac{\Bar{w}_{\widehat{i}}}{\max_{\widehat{i}}\{\Bar{w}_{\widehat{i}}\}}}.
    \end{equation*}
    Its cost is computed as
    \begin{equation*}
        c_{\widehat{i}}=\sqrt{\frac{\max_{\widehat{i}}\{\Bar{w}_{\widehat{i}}\}}{\Bar{w}_{\widehat{i}}}}.
    \end{equation*}
    We use the square root and the cubic root to avoid a constant $w_{\widehat{i}}c_{\widehat{i}}$. If both are square root, we have $w_{\widehat{i}}c_{\widehat{i}}=1$ for all $\widehat{i}$. We use the cubic root on the weight so that $w_{\widehat{i}}c_{\widehat{i}}>1$.
    For $E_{\widehat{i}}\notin \mathcal{E}$, we set $c_{\widehat{i}}=2\max_{\widehat{i}}\{w_{\widehat{i}}\}$.
    The result is shown in Figure \ref{HARCCC1}.
    \begin{figure}[H]
        \centering\includegraphics[width=8cm]{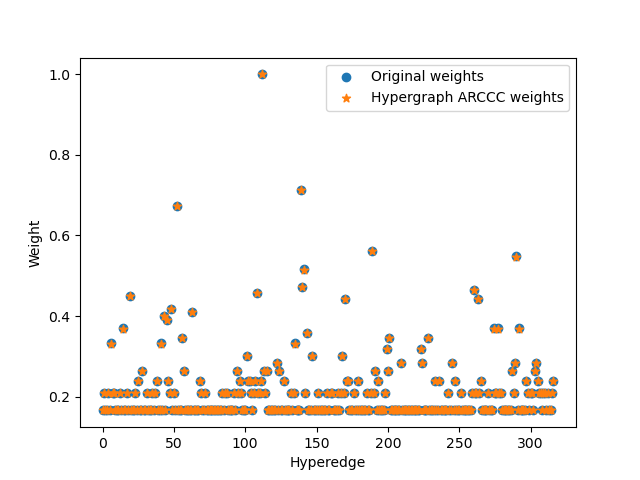}
        \centering\includegraphics[width=8cm]{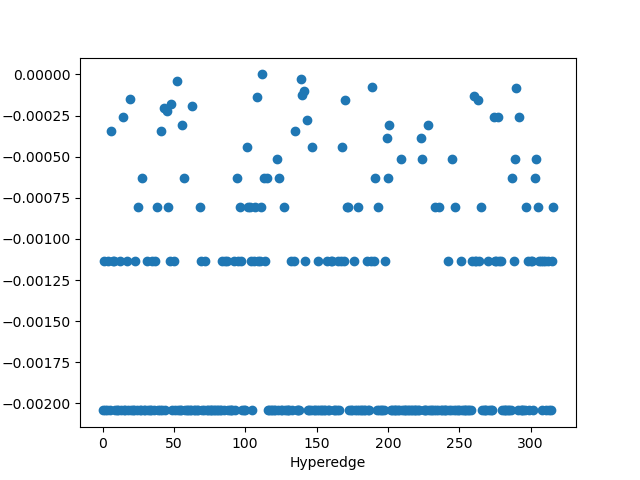}
        \caption{Results of the hypergraph ARCCC Problem on the Email-Enron dataset.}       
    \label{HARCCC1}  
    \end{figure}
    For the left-hand  side, the cycle is the $w_{\widehat{i}}$, where $E_{\widehat{i}}\notin\mathcal{E}$ is omitted. The star is the solution of the hypergraph ARCCC problem after the following modification process. We omit the value that is less than $10^{-4}$, and the remains are greater than $10^{-2}$. There is no value located in this interval, which yields a sharp gap naturally. We normalized the weight by a linear map so that the maximum is $1$. It is shown that the set of the remaining hyperedges is just $\mathcal{E}$, and the weights is quite near with the real weights. The right side is the relative error.
    We also make an example on the Contact-Primary-School dataset. 
    \begin{figure}[H]
        \centering\includegraphics[width=8cm]{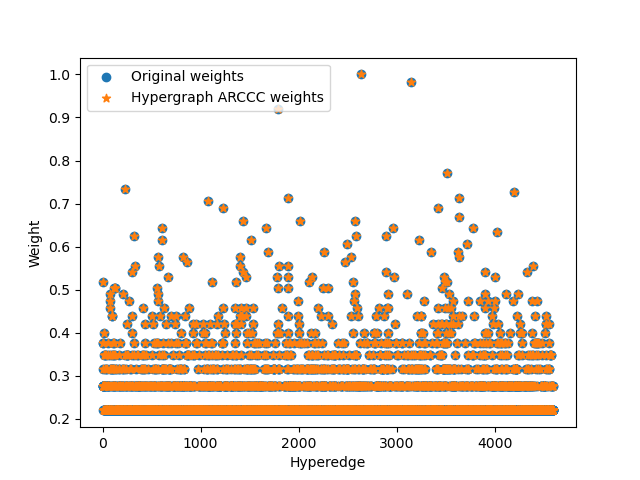}
        \centering\includegraphics[width=8cm]{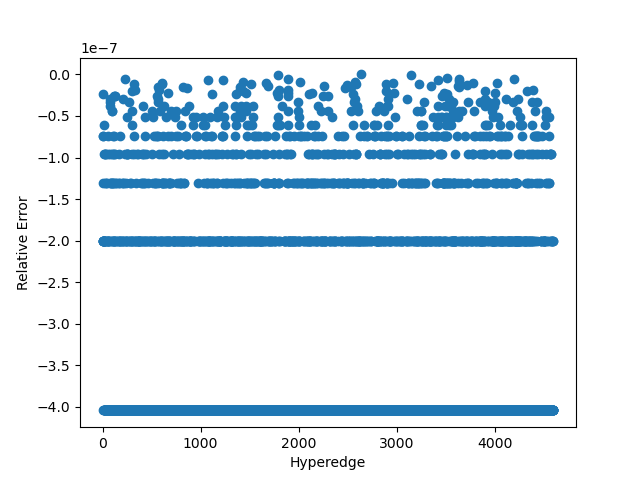}
        \caption{Results of the hypergraph ARCCC Problem on the Contact-Primary-School dataset.}       
    \label{HARCCC2}  
    \end{figure}
    The process of the following modification on the solution of the hypergraph ARCCC problem is the almost the same. The only different thing is that the gap is from $10^{-3}$ to $10^{-1}$.
    It can be shown that the hypergraph ARCCC problem can capture the importance of each hyperedge efficiently.

\section{Application in Dimensionality Reduction: Hypergraph Laplacian Eigenmap}
    The Laplacian eigenmap algorithm \cite{2001Laplacian,2003Laplacian} can be generalized to the hypergraph Laplacian eigenmap algorithm in our tensor product structure. The eigenmap step can also serve as a data preprocess step if the following clustering or classification algorithms use the linear distance rather than a kernel distance. 
    
    \begin{algorithm}[h]
    \caption{Hypergraph Laplacian Eigenmap} 
    \hspace*{0.02in} {\bf Input:} 
    Dataset $\{x_1,\dots,x_n\}\in\mathbb{R}^{n\times p}$.\\
    \hspace*{0.02in} {\bf Output:} 
    the dimension reduced dataset: $\{y_1,\dots,y_n\}\in\mathbb{R}^{n\times q}$, $q\ll p$.
    \begin{algorithmic}[1]
    \State Compute $\Vert x_i-x_j\Vert$ for all $i\neq j$.
    \State Find the $m$ nearest neighbors $n_{i1},\dots,n_{im}$ of $x_i$ for $i=1,\dots,n$.
    \State Construct an $(m+1)$-uniform hypergraph with $\mathcal{E}=\left\{E_i=\{n_{i1},\dots,n_{im},x_i\}|i\in[n]\right\}$.
    \State Construct the Laplacian tensor and compute its eigenvalue and eigenvector.
    \State Let $\{v_i\}_{i=1}^q$ be the eigenvectors corresponding to the $q$ smallest non-zero eigenvalues, then  $\{y_1,\dots,y_n\}=\{v_1,\dots,y_q\}^\top$.
    \State \Return $\{y_1,\dots,y_n\}$.
    \end{algorithmic}
    \end{algorithm}


    We simply use the Laplacian tensor to compute the eigenvector rather than the normalized Laplacian matrix in the classic graph Laplacian eigenmap algorithm in \cite{2001Laplacian,2003Laplacian}. We applied the original Laplacian eigenmap algorithm and our hypergraph Laplacian eigenmap algorithm and then perform a standard $K$-means clustering algorithm after the dimensionality reduction. The performance is further measured by the normalized mutual information and the adjusted Rand score. Both the $K$-means algorithm and the two measures are from the package Scikit-learn in python. Before dimensionality reduction, a normalization step for each attribute of the dataset is performed. All the datasets are from KEEL dataset repository.
    The result is shown below. The Breast Cancer Wisconsin (Diagnostic) dataset has 33 attributes, we choose $m=7$ nearest neighbors of each vertex, leading to an $8$-uniform hypergraph. In the graph Laplacian eigenmap, this means each vertex is connected to $7$ nearest vertices. We use the $0-1$ weight rather than the heat kernel weight to avoid the choice of the heat hyperparameter $t$ in \cite{2001Laplacian,2003Laplacian}. 
    
    We use HLE as the abbreviation for the hypergraph Laplacian eigenmap, and GLE for the graph Laplacian eigenmap. The horizontal coordinate is the reduced dimension. 

    \begin{figure}[H]
        \centering\includegraphics[width=8cm]{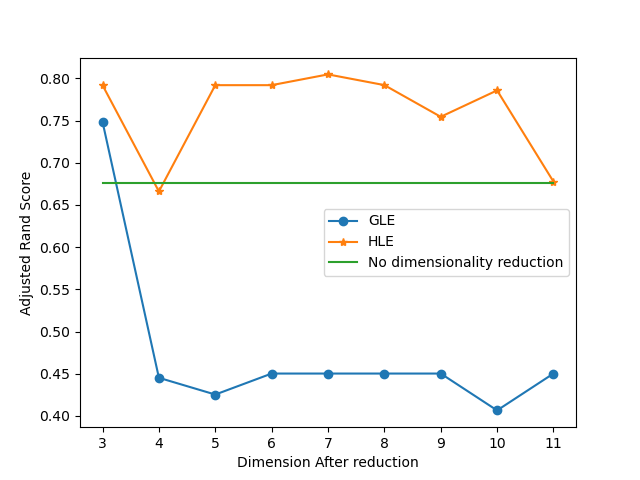}
        \centering\includegraphics[width=8cm]{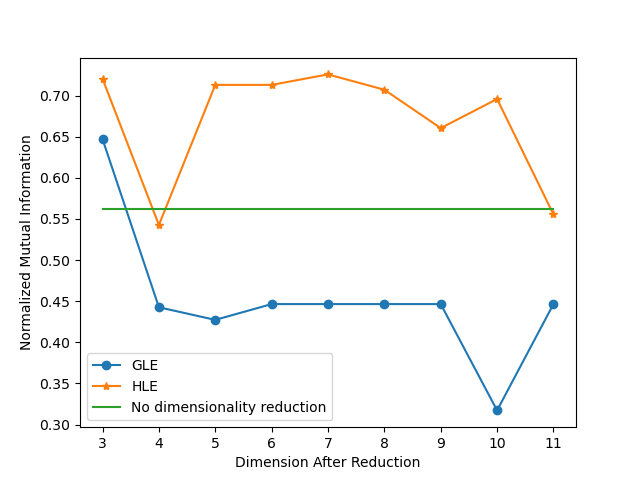}
        \caption{Results on the Breast Cancer Wisconsin (Diagnostic) dataset.}       
    \label{fig1}  
    \end{figure} 

    The green horizontal line in the figures is the result with no dimensionality reduction. We simply perform the $K$-means algorithm on the normalized datasets and use this result as a baseline. Our hypergraph Laplacian eigenmap algorithm has better performance and better stability than the classic Laplacian eigenmap algorithm. The same result happens on the Dermatology dataset and the Movement of Libras dataset, which are shown in Figure \ref{fig2} and \ref{fig3}. 
    \begin{figure}[H]
        \centering\includegraphics[width=8cm]{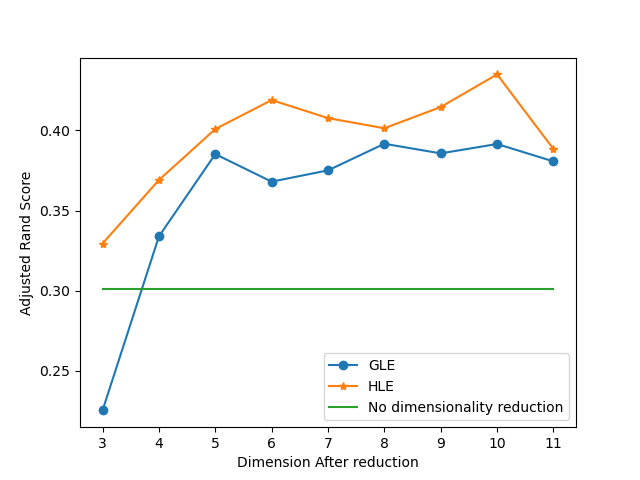}
        \centering\includegraphics[width=8cm]{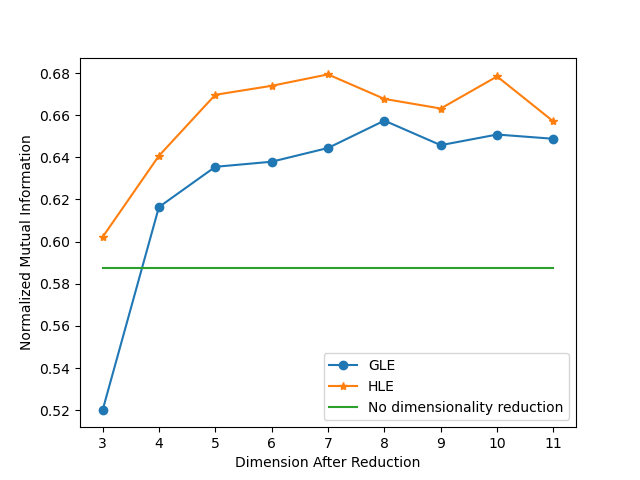}
        \caption{Results on the Movement of Libras dataset.}       
    \label{fig2}  
    \end{figure}
    \begin{figure}[H]
        \centering\includegraphics[width=8cm]{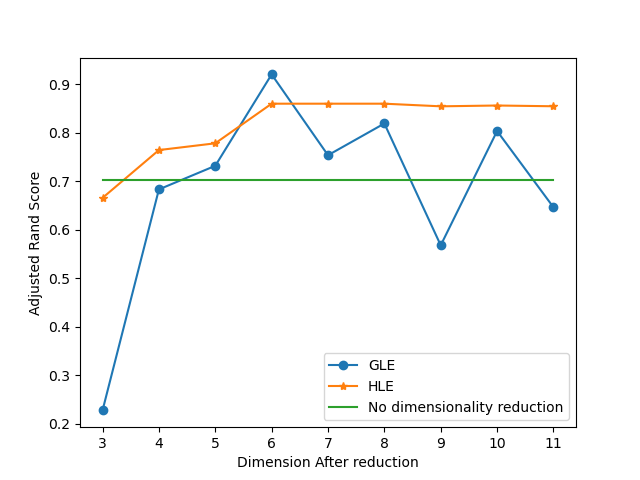}
        \centering\includegraphics[width=8cm]{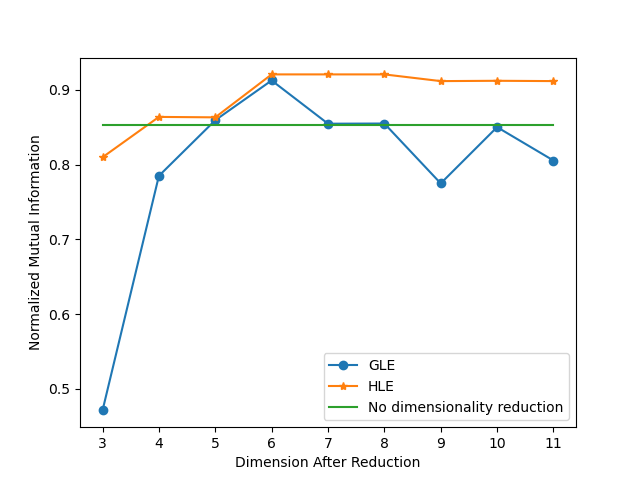}
        \caption{Results on the Dermatology dataset.}       
    \label{fig3}  
    \end{figure}

    In the third section, we mentioned that our structure is the same as the clique reduction technique in some cases. This means each hyperedge contains more edges. A hyperedge containing $(m+1)$ vertices actually contains $\frac{1}{2}m(m+1)$ edges in the view of the clique reduction.  On the other hand, in our algorithm we remove the normalized item $x^\top Dx=I$. To show that the improvement is neither due to the removal of the normalized term, nor due to the increase of edges, there is a detailed comparison. 

    To figure out the influence of increasing edges, we firstly increase the number of nearest neighbors in the graph Laplacian eigenmap algorithm to $\frac{1}{2}m(m+1)=28$.
    \begin{figure}[H]
        \centering\includegraphics[width=8cm]{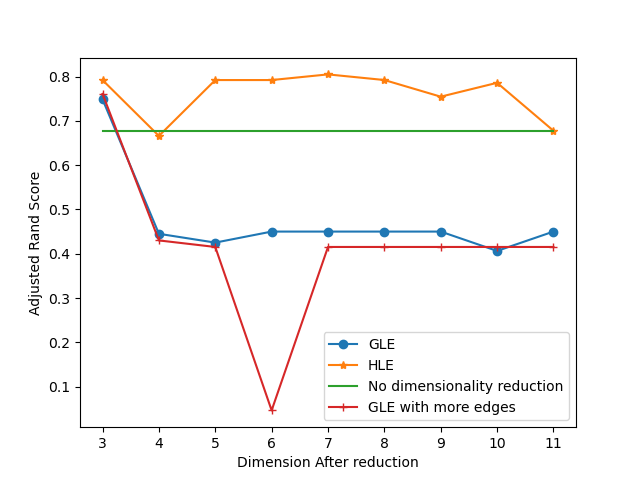}
        \centering\includegraphics[width=8cm]{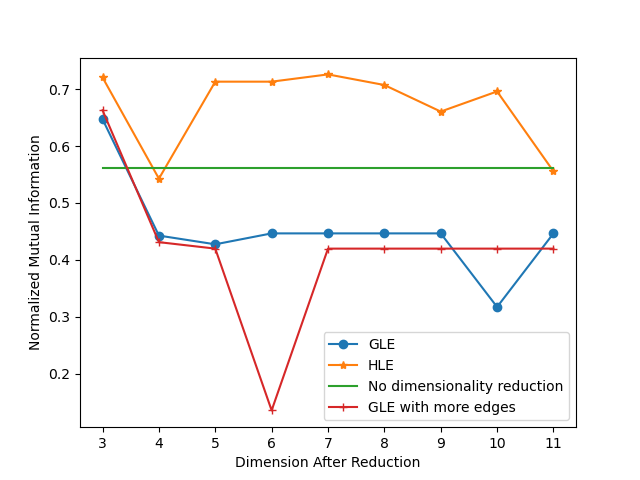}
        \caption{Influence of more edges on the Breast Cancer Wisconsin (Diagnostic) dataset.}       
    \label{fig4}  
    \end{figure}
    It can be seen that increasing the number of edges does not always improve the performance. This can also be seen in the results on the other two datasets, especially on the Movement of Libras dataset.
    \begin{figure}[H]
        \centering\includegraphics[width=7cm]{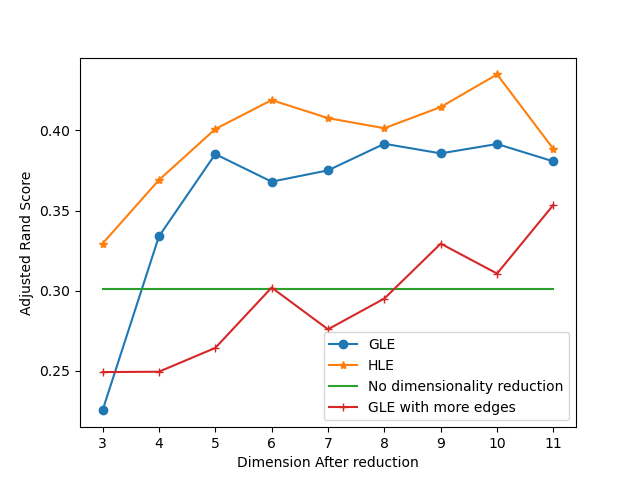}
        \centering\includegraphics[width=7cm]{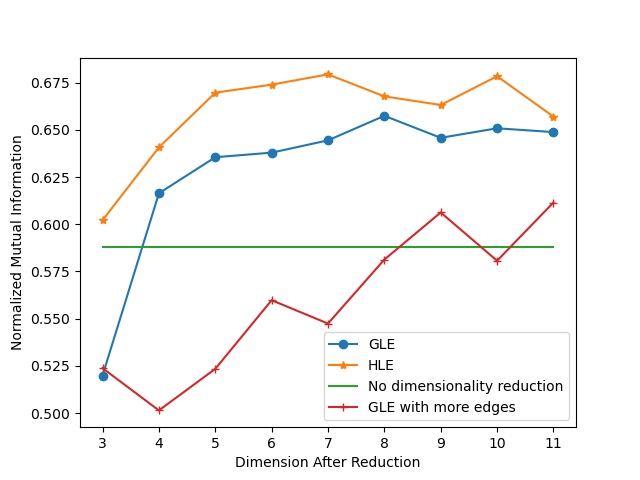}
        \caption{Influence of more edges on Movement of Libras dataset.}       
    \label{fig5}  
    \end{figure}

    \begin{figure}[H]
        \centering\includegraphics[width=7cm]{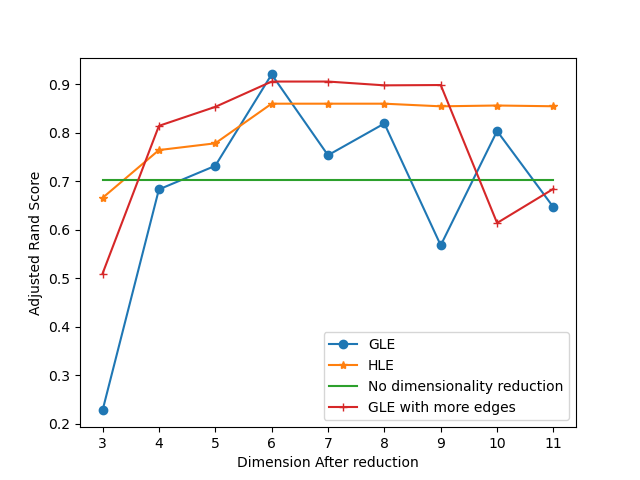}
        \centering\includegraphics[width=7cm]{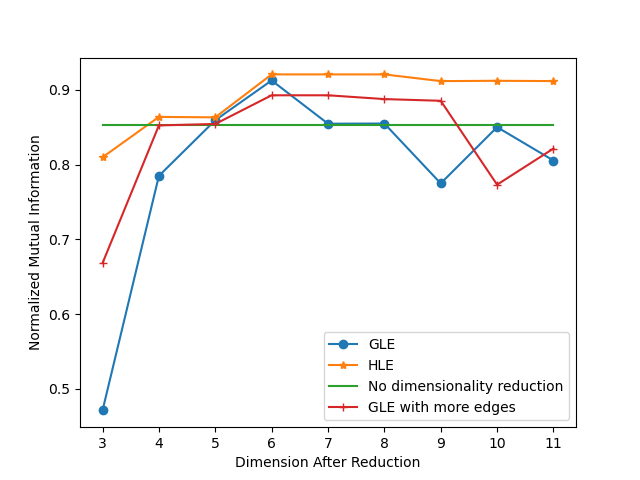}
        \caption{Influence of more edges on the Dermatology dataset.}       
    \label{fig6}  
    \end{figure}

    We also need to clarify the influence of the normalization item. In the following experiments we remove the normalization item. We also make an example of these two factors to figure out whether the improvement is due to a combined action.

    \begin{figure}[H]
        \centering\includegraphics[width=7cm] {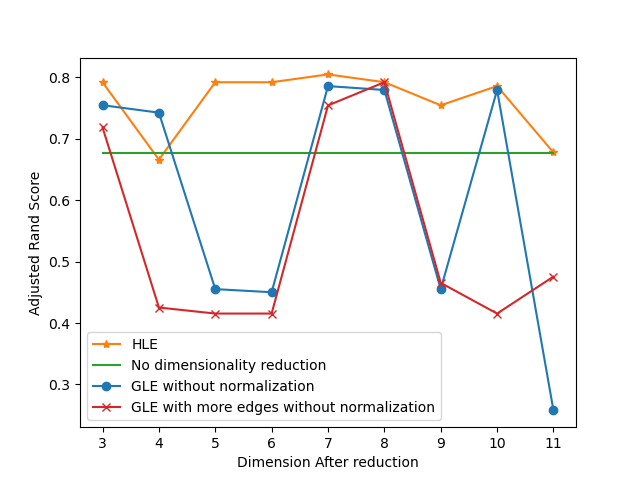}
        \centering\includegraphics[width=7cm]{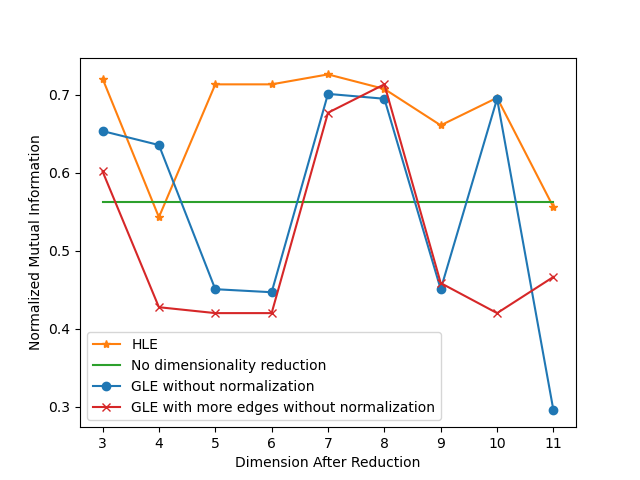}
        \caption{Influence of the normalization term on the Breast Cancer Wisconsin (Diagnostic) dataset.}       
    \label{fig7}  
    \end{figure}

    \begin{figure}[H]
        \centering\includegraphics[width=8cm]{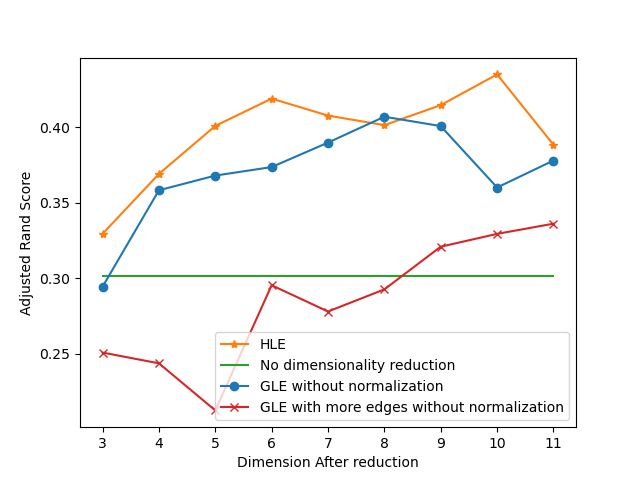}
        \centering\includegraphics[width=8cm]{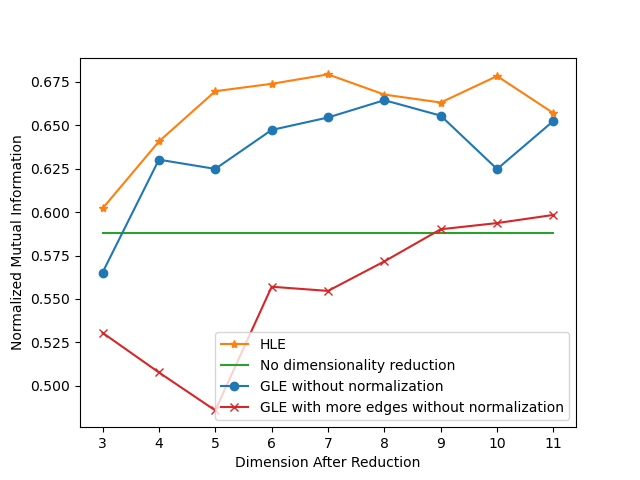}
        \caption{Influence of the normalization term on the Movement of Libras dataset.}       
    \label{fig8}  
    \end{figure}

    \begin{figure}[H]
        \centering\includegraphics[width=8cm]{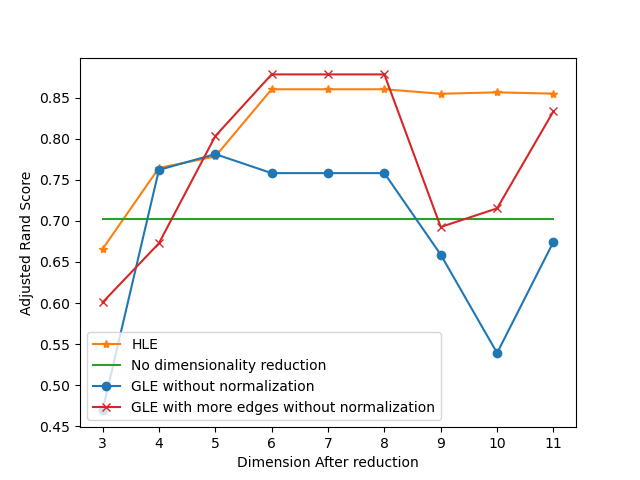}
        \centering\includegraphics[width=8cm]{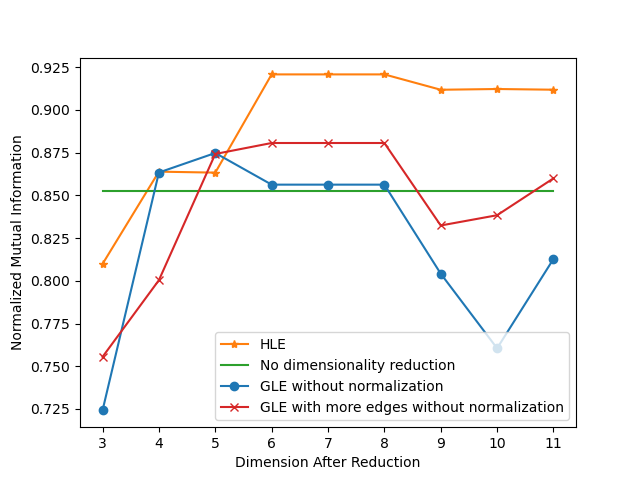}
        \caption{Influence of the normalization term on the Dermatology dataset.}       
    \label{fig9}  
    \end{figure}

    It can be seen that removing the normalization item is not so good as the hypergraph Laplacian eigenmap algorithm. Moreover, if we increase the number of edges in the graph Laplacian eigenmap algorithms without the normalization term, the result is still not so good as our hypergraph Laplacian eigenmap algorithm. This shows the improvement is not due to either one of these two factors, or a combined action. In conclusion, our structure provides a hypergraph Laplacian eigenmap algorithm, and have better performance than the classic graph Laplacian eigenmap algorithm.

\section*{Acknowledgements}
 The authors would like to thank the handling editor and
 the referee for their detailed comments.


{\small
\bibliographystyle{siam}
\bibliography{graph} 
}

\end{document}